\documentclass[12pt,leqno]{article}
\usepackage{amsfonts}
\usepackage{graphicx, amsmath, amsthm, amsfonts, amssymb, color, ulem}
\usepackage{mathrsfs, float}
\usepackage{cite}
\newtheorem{thm}{Theorem}[section]

\newtheorem{ass}[thm]{Assumption}
\newtheorem{lem}[thm]{Lemma}

\newtheorem{rem}[thm]{Remark}
\theoremstyle{definition}
\newtheorem{defn}{Definition}[section]

\definecolor{wco}{rgb}{0.5,0.2,0.3}
\renewcommand{\bar}{\overline}

\numberwithin{equation}{section}

\renewcommand{\hat}{\widehat}
\renewcommand{\tilde}{\widetilde}

\title{\bf Strong convergence of tamed theta scheme for superlinearly growing McKean-Vlasov NSDDEs driven by fractional Brownian motions\thanks{Supported by Jiangxi Provincial Natural Science Foundation(No., 20252BAC250002, GJJ2404201)}}
\author
{ {\bf Li Tan}$^{\tt a, b}$, {\bf Shizhong Hu}$^{\tt a, c}$, {\bf Shengrong Wang}$^{\tt d}$\thanks{Contact E-mail address: 451087433@qq.com}
	\\[0.5ex]
$^{\tt a}$School of Data Science and Statistics, Jiangxi University of\\ Finance and Economics,
Nanchang, Jiangxi, 330013, China \\
$^{\tt b}$Key Laboratory of Data Science in Finance and Economics,
Jiangxi\\ University of Finance and Economics, Nanchang, Jiangxi, 330013, China\\
$^{\tt c}$College of Modern Economics and Management, Jiangxi University of\\ Finance and Economics,
Nanchang, Jiangxi, 330013, China \\
$^{\tt d}$School of Science, East China JiaoTong University,\\ Nanchang, Jiangxi, 330013, China}
     \date{\today}
\begin{document}
\maketitle %
\begin{abstract}
In this article, we study the McKean-Vlasov neutral stochastic differential delay equations driven by fractional Brownian motion with super-linearly growing coefficients, where the Hurst exponent $H\in(1/2,1)$. The existence and uniqueness of the exact solution were shown by the Picard iteration. Besides, we propose a tamed theta Euler-Maruyama scheme for this equation, analyzed the moment boundness and propagation of chaos etc. Moreover, the convergence rate of the numerical scheme is established.

\end{abstract}

\section{introduction}
Stochastic differential equations (SDEs) whose coefficients depend on the law of the solution (also called McKean-Vlasov SDEs), were first investigated by McKean \cite{M66}, whose work drew inspiration from Kac's programme in Kinetic Theory \cite{K56}. Since then, McKean-Vlasov SDEs have evolved into a major area of research in stochastic analysis. To approximate the McKean-Vlasov SDE limit, Bossy and Talay \cite{BT96,BT97} proposed a system of $n$ weakly interacting particles---replicating the statistical behavior of the limit system, and systematically studied the numerical approximation of its solutions. Building on this foundation, the existence and uniqueness of McKean-Vlasov SDEs were discussed in \cite{W18,DST19}, while the well-posedness and the propagation of chaos for McKean-Vlasov SDEs were investigated by Lacker \cite{L18}. Given the rarity of explicit solutions for McKean-Vlasov SDEs, numerical approximation methods have become indispensable. A variety of schemes have been developed to handle different coefficients behaviors, for example the implicit Euler-Maruyama (EM) scheme \cite{DE20,CD22}, the adaptive EM scheme \cite{RS22,FG20}, the tamed EM scheme \cite{KNR22,GGHY24}, and truncated EM scheme \cite{CLL24}.

In recent years, research attention has increasingly shifted to McKean-Vlasov SDEs driven by fractional Brownian motion---a generalization of standard Brownian motion with unique self-similarity and long-memory properties, making it well-suited for modeling phenomena in atmospheric sciences, financial engineering, and other fields. These equations typically take the form:
\begin{align}\label{MVdrivenbyfBmmoxing}
  {\rm d}X_{t}=b(X_{t},\mathcal{L}(X_{t})){\rm d}t+\sigma(\mathcal{L}(X_{t})){\rm d}B_{t}^H,
\end{align}
 where the coefficients $b:\mathbb{R}^d\times\mathcal{P}_{\theta}(\mathbb{R}^d)\to \mathbb{R}^d, \sigma:\mathcal{P}_{\theta}(\mathbb{R}^d)\to\mathbb{R}^d\otimes\mathbb{R}^d$ are Borel-measurable, here, $ \mathcal{P}_{\theta}(\mathbb{R}^d)$ is the set of probability measures on $\mathbb{R}^d$ with finite $\theta$-th moment. $\mathcal{L}(X_{t})$ is the law of $X$ at $t$ and  $B^H=\{B_{t}^H, t\in[0,T]\}$ is a fractional Brownian motion with Hurst exponent $H\in(0,1)$ defined on the probability space $(\Omega,\mathcal{F},\mathbb{P})$. As a centered Gaussian process, $B_H$ has the covariance function
\begin{align}\label{covariance}
R_{H}(t,s)=\mathbb{E}(B_{t}^HB_{s}^H)=\frac{1}{2}(t^{2H}+s^{2H}-|t-s|^{2H}),\ \ \forall t,s\in[0,T].
\end{align}
Notably, when $H=1/2$, the fractional Brownian motion $\{B_{t}^H\}_{t\geq 0}$ reduces to a standard Brownian motion. The well-posedness of such equations was first proven by Fan et al.\cite{FHS22}. Subsequent research have further expanded the theoretical framework of this model. For example, under the assumption that the coefficients satisfy globally Lipschitz conditions, He et al. \cite{HGZ24} established the propagation of chaos for the model described in \eqref{MVdrivenbyfBmmoxing} and analyzed the strong convergence of the EM method.

A critical gap in current literature is the lack of studies on McKean-Vlasov SDEs driven by fractional Brownian motion with super-linearly growing coefficients. One notable exception is Shen et al.\cite{SXW22} who showed that under specific averaging conditions, solutions to such equations can be approximated by solutions of an associated averaged equation. More recently, Wang et al. \cite{WXT25} and Gao et al. \cite{GGL25} focused on the EM scheme for delayed McKean-Vlasov SDEs driven by fractional Brownian motion, but their work only considered super-linear growth in the delay term of the drift coefficient, not the coefficient itself.

To capture the influence of past states on current dynamics, stochastic delay differential equations (SDDEs) and neutral stochastic differential delay equations (NSDDEs) are essential tools in stochastic calculus, with broad applications in modeling natural and engineered systems, see for example \cite{JY17,BY15,LH20}. A key challenge arises when NSDDEs exhibit super-linear growth: the classical EM scheme often fails to converge. To address this, a tamed EM scheme was introduced to handle super-linearly growing coefficients. For instance, Deng et al. \cite{DFF21} proposed two types of explicit tamed EM schemes for NSDDEs with super-linearly growing drift and diffusion coefficients and proved strong convergence for each. Cui et al. \cite{CLL21}  studied the tamed EM methods to approximate the solutions of McKean-Vlasov NSDDEs with super-linear drift, establishing the convergence rate. Additional insights into tamed EM schemes for NSDDEs can be found in \cite{LH20, ZJ19}. Notably, Tan and Yuan \cite{TY18} proposed a  modified tamed theta EM scheme and demonstrated the convergence of the numerical solutions, providing a key motivation for the present work.

Drawing inspiration from the existing literatures, we extend the model \eqref{MVdrivenbyfBmmoxing} to accommodate coefficients with super-linear growth. We propose a tamed theta EM scheme for the McKean-Vlasov NSDDEs, which takes the following form:
\begin{align}\label{benwenmoxing}
  {\rm d}(X_{t}-D(X_{t-\tau}))=b(X_{t},X_{t-\tau},\mathcal{L}(X_{t})){\rm d}t+\sigma(\mathcal{L}(X_{t})){\rm d}B_{t}^H,
\end{align}
where $\tau>0$ denotes the delay.

This paper's main contributions are outlined below:
\begin{itemize}
\item Unlike the prior work (e.g., \cite{WXT25, GGL25}) that only considers super-linear growth in the delay term, we analyze McKean-Vlasov equations driven by fractional Brownian motion where the drift coefficient itself exhibits super-linear growth. We prove the core properties of such equations, representing a significant advance in this field.

\item  By a delicate analysis involving the fractional It\^{o} formula and intricate inequality techniques, we investigate the properties of the solution, the propagation of chaos in associated particle systems, and the convergence of the proposed numerical scheme.
\end{itemize}

The rest of the paper is organized as follows. Section 2 presents the mathematical preliminaries, including key definitions and lemmas for McKean-Vlasov NSDDEs with super-linear growth driven by fractional Brownian motion. Section 3 establishes the existence and uniqueness of the solution. In Section 4, we analyze the propagation of chaos in the particle system associated with our model. Finally, Section 5 introduces the tamed theta EM scheme and derives its strong convergence rate.

\section{Preliminaries}
Throughout this paper, let $(\Omega,\mathcal{F},\mathbb{P})$ be a complete probability space with a filtration $\{\mathcal{F}_t\}_{t\geq 0}$ satisfying the usual conditions (i.e., it is increasing and right continuous while $\mathcal{F}_{0}$ contains all $\mathbb{P}$-null sets). We use $\lvert \cdot \rvert$ and $\langle \cdot,\cdot\rangle$ for the Euclidean norm and inner product, respectively, and let $a\wedge b:=\min(a,b)$ and $a\vee b:=\max(a,b)$. Let $\tau>0$ be a constant and denote by $\mathcal{C}([-\tau,0];\mathbb{R}^d)$ the space of all continuous functions $\varphi$ from $[-\tau,0]$ to $\mathbb{R}^d$ with the norm $\parallel \varphi\parallel=\sup_{-\tau\leq t\leq 0}|\varphi(t)|$. For $p>0$, $L_{\mathcal{F}_0}^p(\Omega;\mathbb{R}^d)$ represents the family of $\mathcal{F}_0$-measurable random variable with $\mathbb{E}|\cdot|^p<\infty$. Let $\mathcal{P}_{\theta}(\mathbb{R}^d), \theta\ge 1$ be the space of probability measures on $\mathbb{R}^d$ with finite $\theta$-th moment. Formally, it is defined as
\begin{align*}
  \mathcal{P}_{\theta}(\mathbb{R}^d)=\bigg\{\mu\in\mathcal{P}(\mathbb{R}^d):\int_{\mathbb{R}^d}|x|^{\theta}\mu({\rm d}x)<\infty\bigg\},
\end{align*}
where $\mathcal{P}(\mathbb{R}^d)$ denotes the family of all probability measures on $(\mathbb{R}^d,\mathcal{B}(\mathbb{R}^d))$, here, $\mathcal{B}(\mathbb{R}^d)$ is the Borel $\sigma$-field over $\mathbb{R}^d$. For any $\mu\in\mathcal{P}_{\theta}(\mathbb{R}^d)$, define the $\theta$-th moment norm of $\mu$ as
\begin{align*}
  \mathcal{W}_{\theta}(\mu)=\bigg(\int_{\mathbb{R}^d}|x|^{\theta}\mu({\rm d}x)\bigg)^{1/\theta}.
\end{align*}
Next, for two probability measures $\mu,\nu\in\mathcal{P}_{\theta}(\mathbb{R}^d)$, the Wasserstein distance between $\mu$ and $\nu$ is defined as
\begin{align*}
\mathbb{W}_{\theta}(\mu,\nu)=\bigg(\inf_{\pi\in\Pi(\mu,\nu)} \int_{\mathbb{R}^d\times\mathbb{R}^d}|x-y|^{\theta}\pi({\rm d}x,{\rm d}y)\bigg)^{1/\theta},
\end{align*}
where $\Pi(\mu,\nu)$---the set of coupling measures between $\mu$ and $\nu$, consists of all probability measures on $\mathbb{R}^d\times\mathbb{R}^d$ whose marginal measures are $\mu$ and $\nu$, respectively. Precisely, a measure $\pi\in\Pi(\mu,\nu)$ satisfies $\pi(A\times\mathbb{R}^d)=\mu(A)$ and $\pi(\mathbb{R}^d\times A)=\nu(A)$ for every Borel set $A\subseteq\mathbb{R}^d$. The space $\mathcal{P}_2(\mathbb{R}^d)$ is a Polish space when equipped with the $L^{2}$-Wasserstein distance \cite[Definition 6.1]{V09}.

Let $\{B_{t}^H\}_{t\in[0,T]}$ be a $d$-dimensional fractional Brownian motion with Hurst exponent $H\in(0,1)$. For Hurst exponents $1/2<H<1$, we define the kernel function
\begin{align}\label{phideshizi}
  \phi(s,t)=\phi_{H}(s,t)=H(2H-1)|s-t|^{2H-2},\ \ s,t\in\mathbb{R},
\end{align}
and recall that for $s,t>0$, this kernel satisfies the integral identity
\begin{align*}
  \int_{0}^t\int_{0}^s\phi(u,v){\rm d}u{\rm d}v=R_{H}(t,s),
\end{align*}
where $R_{H}(t,s)$ is defined as in \eqref{covariance}. Let $\mathcal{S}(\mathbb{R})$ be the Schwartz space (space of rapidly decreasing smooth functions) on $R$. For any $f\in \mathcal{S}(\mathbb{R})$, we define its squared norm with respect to the kernel $\phi$ as
\begin{align*}
  \| f\|_{H}^2:=\int_{\mathbb{R}}\int_{\mathbb{R}}f(s)f(t)\phi(s,t){\rm d}s{\rm d}t<\infty,
\end{align*}
where $\phi$ is the kernel defined in \eqref{phideshizi}. Next, we equip $\mathcal{S}(\mathbb{R})$ with an inner product induced by $\phi$ as
\begin{align*}
  \langle f,g\rangle_{H}:=\int_{\mathbb{R}}\int_{\mathbb{R}}f(s)g(t)\phi(s,t){\rm d}s{\rm d}t,\ \ f,g\in\mathcal{S}(\mathbb{R}).
\end{align*}
The completion of $\mathcal{S}(\mathbb{R})$ with respect to this inner product is denoted by $L_{\phi}^2(\mathbb{R})$. For the half line $\mathbb{R}_{+}:=[0,+\infty)$, we further define the space $L_{\phi}^2(\mathbb{R}_{+})$ as the set of measurable functions $g: \mathbb{R}_{+}\rightarrow\mathbb{R}$ such that their squared $\phi$-norm is finite:
\begin{align*}
 L_{\phi}^2(\mathbb{R}_{+})=\bigg\{g|g:\mathbb{R}_{+}\to\mathbb{R},\ \ \ |g|_{\phi}^2:=\int_{0}^{\infty}\int_{0}^{\infty}g(s)g(t)\phi(s,t){\rm d}s{\rm d}t<\infty\bigg\}.
\end{align*}
Here, $L_{\phi}^2(\mathbb{R}_{+})$ can be understood as the natural restriction of $L_{\phi}^2(\mathbb{R})$ to functions supported on $\mathbb{R}_{+}$, and it is also a Hilbert space under the inner product
\begin{align*}
\langle g_1, g_2\rangle_{\phi}:=\int_{0}^{\infty}\int_{0}^{\infty}g_1(s)g_2(t)\phi(s,t){\rm d}s{\rm d}t.
\end{align*}
\begin{defn}\cite[Definition 3.5.1]{BH008}\label{MalliavinDerivative}
 {\rm Let $g\in L_{\phi}^2(\mathbb{R})$, and $\Phi_{g}$ denote a random variable induced by $g$, i.e., $\Phi_{g}=\int_{\mathbb{R}}g(s){\rm d}B_s^H$. For $F\in L^p$, the $\phi$-derivative of $F$ in the direction of $\Phi_{g}$ is defined as}
 \begin{align*}
   D_{\Phi_{g}}F(\omega)=\lim_{\delta\to 0}\frac{1}{\delta}\bigg\{F\bigg(\omega+\delta\int_{0}^{\cdot}g(u){\rm d}u\bigg)-F(\omega)\bigg\},
 \end{align*}
{\rm  provided the limit exists in $L^p$. Furthermore, if there is a process $(D_{s}^{\phi}F,s\geq 0)$ such that for all $g \in L_{\phi}^2(\mathbb{R})$}
\begin{align*}
 D_{\Phi_{g}}F=\int_{0}^{\infty}D_{s}^{\phi}F\cdot g_{s}{\rm d}s\ \ \ {\rm  almost\  surely},
\end{align*}
then $F$ is said to be $\phi$-differentiable.
\end{defn}
We denote $L_{\phi}(0, T)$ as the set of stochastic processes $F$ on $[0, T]$ that satisfy the following properties: $F \in L_{\phi}(0, T)$ if and only if $\mathbb{E}\|F\|_{H}^{2}< \infty$, $F$ is $\phi$-differentiable, the trace $D_{s}^{\phi}F_{t}$ exists for all $0 \leq s, t \leq T$, $\mathbb{E}\int_{0}^{T}\int_{0}^{T}|D_{s}^{\phi}F_{t}|^{2}\,ds\,dt < \infty$, and for each sequence of partitions $(\pi_{n}, n \in \mathbb{N})$ such that $\|\pi_{n}\| \to 0$ as $n \to \infty$, it holds that
\begin{align*}
\sum_{i,j=0}^{n-1} \mathbb{E}\left[ \int_{t_i^{(n)}}^{t_{i+1}^{(n)}} \int_{t_j^{(n)}}^{t_{j+1}^{(n)}} | D_s^\phi F_{t_i^{(n)}}^\pi D_t^\phi F_{t_j^{(n)}}^\pi - D_s^\phi F_t D_t^\phi F_s | ds dt \right]\to 0,
\end{align*}
and
\begin{align*}
 \mathbb{E}\| F^\pi - F \|_H^2 \to 0, n\to \infty,
\end{align*}
where $\pi_n: 0 = t_0^{(n)} < t_1^{(n)} < \cdots < t_{n-1}^{(n)} < t_n^{(n)} = T.$

Consider the $d$-dimensional SDEs driven by a $d$-dimensional fractional Brownian motion $B_{t}^H$ with Hurst exponent $H\in(1/2, 1)$ as follows
\begin{align}\label{fbmitofunction}
  {\rm d}X_{t}=G(t,X_{t}){\rm d}t+F(t,X_{t}){\rm d}B_{t}^H,
\end{align}
where $G:[0,T]\times\mathbb{R}^d\rightarrow\mathbb{R}^d$ and $F:[0,T]\times\mathbb{R}^d\rightarrow\mathbb{R}^{d\times d}$. Let $C^{2,1}([0,T]\times\mathbb{R}^d;\mathbb{R})$ denote the space of real-valued functions $V(t,x)$ that are twice continuously differentiable in $x\in\mathbb{R}^d$ and once continuously differentiable in $t\in[0,T]$. Based on the fractional It\^{o} formula derived in \cite{BH008,HGZ23}, the differential form of $V(t,X_{t})$ takes the form:
\begin{align*}
 {\rm d}V(t,X_{t})=&\frac{\partial V(t,X_{t})}{\partial t}{\rm d}t+ \frac{\partial V(t,X_{t})}{\partial X}G(t,X_{t}){\rm d}t+\frac{\partial V(t,X_{t})}{\partial X}F(t,X_{t}){\rm d}B_{t}^H\\
 &+\frac{\partial^2 V(t,X_{t})}{\partial X^2}F(t,X_{t})D_{t}^{\phi}X_{t}{\rm d}t,
\end{align*}
where $\frac{\partial V}{\partial X}$ denotes the gradient, $\frac{\partial^2 V}{\partial X^2}$ is the Hessian matrix. Formally, this can also be expressed by the kernel $\phi$ as
\begin{align*}
 {\rm d}V(t,X_{t})=&\frac{\partial V(t,X_{t})}{\partial t}{\rm d}t+ \frac{\partial V(t,X_{t})}{\partial X}G({t,X_{t}}){\rm d}t+\frac{\partial V(t,X_{t})}{\partial X}F(t,X_{t}){\rm d}B_{t}^H\\
 &+\frac{\partial^2 V(t,X_{t})}{\partial X^2}F(t,X_{t})\int_{0}^t\phi(t,r)F(r,X_{r}){\rm d}r{\rm d}t.
\end{align*}
Additional discussions on the fractional It\^{o} formula can be found in \cite{LRQ10, N03}. We summarize the result below.

\begin{lem}(Fractional It\^{o} formula)\label{Itoformula}
 {\rm Let $X_{t}$ be the solution of \eqref{fbmitofunction} with $1/2<H<1$, and assume that $\mathbb{E}\bigg[\int_{0}^T\left\lvert F(s,X_{s})D_{s}^{\phi}\int_{0}^sF(u,X_{u}){\rm d}B_{u}^H\right\rvert{\rm d}s\bigg]<\infty$, $\frac{\partial V(s,X_s)}{\partial X}F(s,X_s)\in L_{\phi}(0,T)$ for $V\in C^{2,1}([0,T]\times\mathbb{R}^d;\mathbb{R})$. Then we have }
\begin{align*}
&V(t,X_t)\\
=&V(0, X_0)+\int_0^t\bigg[\frac{\partial V(s,X_s)}{\partial s}+\frac{\partial V(s,X_s)}{\partial X}G(s,X_s)\\
&+H(2H-1)\text{{\rm tr}}\left(\frac{\partial^2 V(s,X_s)}{\partial X^2}\int_0^s(s-r)^{2H-2}F(r,X_r)F^T(r,X_r){\rm d}r\right)\bigg]{\rm d}s\\
&+\int_0^t\frac{\partial V(s,X_s)}{\partial X}F(s,X_s){\rm d}B_s^H,
\end{align*}
{\rm where $F^T$ is the transpose of $F$, and $\text{{\rm tr}}(\cdot)$ denotes the matrix trace. }
\end{lem}

Since the kernel function $\phi$ is only valid for $1/2<H<1$, we restrict our analysis to Hurst exponents satisfying $1/2<H<1$ throughout this paper. In this work, we focus on the following $d$-dimensional McKean-Vlasov NSDDEs driven by fractional Brownian motion as follows:
\begin{align}\label{MVdrivenbyfBm}
\left\{
\begin{array}{ll}
  {\rm d}(X_{t}-D(X_{t-\tau}))=b(X_{t},X_{t-\tau},\mathcal{L}_{X_{t}}){\rm d}t+\sigma(\mathcal{L}_{X_{t}}){\rm d}B_{t}^H,\ \ t>0, \\
  X_{t}=\xi=\{\xi(t):-\tau\leq t\leq 0\}\in L_{\mathcal{F}_0}^{{p}}([-\tau,0];\mathbb{R}^d),\ {p}\geq1
\end{array}
\right.
\end{align}
where $b:\mathbb{R}^d\times\mathbb{R}^d\times\mathcal{P}_2(\mathbb{R}^d)\to \mathbb{R}^d$, $\sigma:\mathcal{P}_{2}(\mathbb{R}^d)\to \mathbb{R}^{d\times d}$, $D:\mathbb{R}^d\to\mathbb{R}^d$, $\mathcal{L}_{X_{t}}$ denotes the law of $X_{t}$. Throughout the paper, we shall denote by $C$ a generic positive constant that may depend on $T, H, \xi, l, \lambda, K_i$ etc., but independent of $n, m, s, t, N, \Delta$. The value may change from line to line.

\begin{ass}\label{chushizhixidetiaojian}
{\rm For any $s,t\in[-\tau,0]$ and $p>0$, there is a positive constant $K_0$ such that }
\begin{align*}
  \mathbb{E}\bigg(\sup_{-\tau\leq s,t\leq 0}|\xi(s)-\xi(t)|^p\bigg)\leq K_{0}|s-t|^p.
\end{align*}
\end{ass}

\begin{ass}\label{zhonglixiangdetiaojian}
  {\rm $D(0)=0$, and there exists a constant $0<\lambda<1$ such that
  \begin{align*}
    |D(x)-D(\bar{x})|\leq \lambda|x-\bar{x}|,
  \end{align*}
for $x, \bar{x}\in\mathbb{R}^d$.}
\end{ass}

\begin{ass}\label{danbiantiaojian}
  {\rm There exist positive constants $l\geq 1$, $ K_2$ and $K_3$ such that}
  \begin{align*}
    \big\langle x-D(y)-\bar{x}+D(\bar{y}),b(x,y,\mu)-b(\bar{x},\bar{y},\nu)\big\rangle\leq K_{2}(|x-\bar{x}|^2+|y-\bar{y}|^2+\mathbb{W}_{2}^2(\mu,\nu)),
  \end{align*}
  {\rm and}
  \begin{align*}
    |b(x,y,\mu)-b(\bar{x},\bar{y},\nu)|\leq K_{3}[(1+|x|^l+|\bar{x}|^l+|y|^l+|\bar{y}|^l)(|x-\bar{x}|+|y-\bar{y}|)+\mathbb{W}_{2}(\mu,\nu)],
  \end{align*}
  {\rm for $x,y,\bar{x},\bar{y}\in\mathbb{R}^d$, $\mu,\nu\in\mathcal{P}_{2}(\mathbb{R}^d)$.}
\end{ass}

\begin{ass}\label{sigmamanzudechushitiaojian}
  {\rm There exist a positive constant $K_4$ such that}
  \begin{align*}
    \parallel \sigma(\mu)-\sigma(\nu)\parallel\leq K_{4}\mathbb{W}_{2}(\mu,\nu).
  \end{align*}
  {\rm And for the initial experience distribution $\delta_0$, there is a positive constant $K_5$ such that}
  \begin{align*}
  |b(0,0,\delta_0)|\vee\parallel\sigma(\delta_0)\parallel\leq K_{5}.
  \end{align*}
\end{ass}

\begin{rem}\label{remarkone}
 {\rm  Let Assumptions \ref{danbiantiaojian}-\ref{sigmamanzudechushitiaojian} hold, then, for any $x,y\in\mathbb{R}^d$, we have}
  \begin{align*}
  |b(x,y,\mu)|\leq &|b(x,y,\mu)-b(0,0,\delta_0)|+|b(0,0,\delta_0)|\\
  \leq &K_{3}[(1+|x|^l+|y|^l)(|x|+|y|)+\mathbb{W}_{2}(\mu,\delta_{0})]+K_{5}\\
  \leq &C(1+|x|^{l+1}+|y|^{l+1}+\mathbb{W}_{2}(\mu,\delta_{0})),
  \end{align*}
  {\rm where $C=2(K_{3}+K_{5})$. Moreover, }
    \begin{align*}
  \|\sigma(\mu)\|\leq &\|\sigma(\mu)-\sigma(\delta_0)\|+\|\sigma(\delta_0)\|  \leq K_{4}\mathbb{W}_{2}(\mu,\delta_{0})+K_{5}\le C(1+\mathbb{W}_{2}(\mu,\delta_{0})),
  \end{align*}
  {\rm where $C=K_{4}\vee K_{5}$. Further, let Assumptions \ref{zhonglixiangdetiaojian}-\ref{sigmamanzudechushitiaojian} hold, then by the Young's inequality, we get }
  \begin{align*}
  \langle x-D(y),b(x,y,\mu)\rangle=&\langle x-D(y)-0+D(0),b(x,y,\mu)-b(0,0,\delta_{0})+b(0,0,\delta_{0})\rangle\\
  \leq &K_{2}(|x|^2+|y|^2+\mathbb{W}_{2}^2(\mu,\delta_{0}))+(|x|+|D(y)|)|b(0,0,\delta_{0})|\\
  \leq &C(1+|x|^2+|y|^2+\mathbb{W}_{2}^2(\mu,\delta_{0})),
  \end{align*}
  {\rm where $C=(K_{2}+1/2)\vee K_{5}^2$.}
\end{rem}
For later use, we state the following lemmas without proof.

\begin{lem}\cite[Lemma 4.1]{M07}\label{Maodebudengshi}
 {\rm Let $p>1$, $\varepsilon>0$ and $x,y\in\mathbb{R}$. Then}
 \begin{align*}
   |x+y|^p\leq \bigg(1+\varepsilon^{\frac{1}{p-1}}\bigg)^{p-1}\bigg(|x|^p+\frac{|y|^p}{\varepsilon}\bigg).
 \end{align*}
 {\rm Especially for $\varepsilon=[(1-a)/a]^{p-1}$ where $0<a<1$, we get}
 \begin{align*}
  |x+y|^p\leq \frac{1}{a^{p-1}}|x|^p+\frac{1}{(1-a)^{p-1}}|y|^p.
 \end{align*}
\end{lem}

\begin{lem}\cite[Lemma 7.1.2]{H06}\label{Hlemma}
{\rm Let $\beta>0$, $\gamma>0$ with $\beta+\gamma>1$. Suppose $a\geq 0$, $b\geq 0$, $u$ is nonnegative and $t^{\gamma-1}u(t)$ is locally integrable on $0\leq t
\le T$. If }
\begin{align*}
u(t)\leq a+b\int_{0}^t(t-s)^{\beta-1}s^{\gamma-1}u(s){\rm d}s, \ a.e.
\end{align*}
{\rm   then}
\begin{align*}
u(t)\leq aE_{\beta,\gamma}((b\Gamma(\beta))^{1/v}t),
\end{align*}
{\rm where $v=\beta+\gamma-1>0$, and $E_{\beta,\gamma}(s)$ denotes the generalized Mittag-Leffler function defined by $E_{\beta,\gamma}(s)=\sum_{m=0}^{\infty}c_{m}s^{mv}$ with $c_{0}=1$, $\frac{c_{m+1}}{c_{m}}=\frac{\Gamma(mv+\gamma)}{\Gamma(mv+\gamma+\beta)}$ for all $m\geq 0$, here $\Gamma(\cdot)$ is the Gamma function.}
\end{lem}

\begin{rem}
{\rm Considering the special situation of Lemma \ref{Hlemma} with $\beta=2H(1/2<H<1)$, $\gamma=1$, $b=1$, it follows immediately that $u(t)$ is bounded, i.e., there exists a positive constant $C$ such that}
\begin{align*}
u(t)\leq C.
\end{align*}
{\rm Since for these parameters, we have $\beta+\gamma=2H+1>1$, so $v=2H>0$. The Mittag-Leffler function in Lemma \ref{Hlemma} reduces to $E_{2H,1}(s)=\sum_{m=0}^{\infty}c_{m}s^{2Hm}$ where $s=[\Gamma(2H)]^{1/2H}t$. This series converges for all $t\in\mathbb{R}$, as shown by analyzing the asymptotic behavior of its coefficients. Using the Stirling's approximation for the Gamma function, we derive}
\begin{align*}
\frac{c_{m+1}}{c_{m}}=\frac{\Gamma(2Hm+1)}{\Gamma(2Hm+1+2H)}\sim\frac{(2H m)^{2H m + 1/2} e^{-2H m}}{(2H m + 2H)^{2H m + 2H + 1/2} e^{-(2H m + 2H)}}.
\end{align*}
{\rm Simplifying further, this ratio satisfies}
\begin{align*}
\frac{c_{m+1}}{c_{m}}\sim (2H m)^{-2H}e^{2H}.
\end{align*}
{\rm Since for $2H>0$ we have }
\begin{align*}
\lim_{m\to\infty}(2H m)^{-2H}e^{2H}=0,
\end{align*}
{\rm the ratio test guarantees that the series $\sum_{m=0}^{\infty}c_{m}s^{2Hm}$ has an infinite radius of convergence. Thus, $E_{2H,1}(\cdot)$ is bounded for all $t\in\mathbb{R}$, implying $u(t)\leq a\cdot E_{2H,1}(\cdot)\leq C.$}
\end{rem}

\section{Existence and Uniqueness}
\begin{thm}\label{Xtnyoujie}
{\rm Let Assumptions \ref{chushizhixidetiaojian}-\ref{sigmamanzudechushitiaojian} hold. Then there exists a unique solution $\{X_t\}_{t\in[-\tau,T]}$ to \eqref{MVdrivenbyfBm}. Moreover, for any $\bar{p}\ge 2$, we have}
  \begin{align}\label{pmoment}
   \mathbb{E}\bigg(\sup_{0\leq t\leq T}|X_{t}|^{\bar{p}}\bigg)\leq C,
  \end{align}
  {\rm where $C$ is a positive constant depending on $T, H, \bar{p}, \lambda, K_i$ and the initial data $\xi$.}
\end{thm}

\begin{proof}
We note that the application of the fractional It\^{o} formula in the following argument is justified by the Lipschitz conditions in Assumptions \ref{zhonglixiangdetiaojian}-\ref{sigmamanzudechushitiaojian}, which ensure the necessary regularity of the Picard iteration sequence $X_t^{(k)}$. Standard theory of Malliavin calculus for SDEs (e.g. \cite{BH008}) guarantees that these iterates are $\phi$-differentiable and satisfy the integrability conditions required by Lemma \ref{Itoformula}. We refer to the proof procedure in \cite{CLL21} and divide the proof into 3 steps.\\
{\bf Step 1. Moment Boundness} For $k\ge 0$ and $t\in[-\tau,0]$, define $X^{(k)}_{t}=\xi(t)$. For $k=0$ and $t\in[0,T]$, define $X^{(0)}_{t}\equiv\xi(0)$. For $k\ge 1$ and $t\in[0,T]$, $X_{t}^{(k)}$ is solved by
\begin{align*}
{\rm d}(X_{t}^{(k)}-D(X^{(k)}_{t-\tau}))=b(X_{t}^{(k)},X_{t-\tau}^{(k)},\mathcal{L}_{X_{t}^{(k-1)}}){\rm d}t+\sigma(\mathcal{L}_{X_{t}^{(k-1)}}){\rm d}B_{t}^H,
\end{align*}
where $\mathcal{L}_{X_{t}^{(0)}}=\mathcal{L}_{\xi_{0}}$. For $i=1$, by Assumption \ref{zhonglixiangdetiaojian} and Lemma \ref{Maodebudengshi}, we get
\begin{align*}
|X_{t}^{(1)}|^{\bar{p}}=&|D(X_{t-\tau}^{(1)})+X_{t}^{(1)}-D(X_{t-\tau}^{(1)})|^{\bar{p}}\\
\leq &\frac{1}{\lambda^{{\bar{p}}-1}}|D(X_{t-\tau}^{(1)})|^{\bar{p}}+\frac{1}{(1-\lambda)^{{\bar{p}}-1}}|X_{t}^{(1)}-D(X_{t-\tau}^{(1)})|^{\bar{p}}\\
\leq &\lambda|X_{t-\tau}^{(1)}|^{\bar{p}}+\frac{1}{(1-\lambda)^{{\bar{p}}-1}}|X_{t}^{(1)}-D(X_{t-\tau}^{(1)})|^{\bar{p}}.
\end{align*}
Then, we conclude that
\begin{align}\label{Xtdebudengshi}
\sup_{0\leq s\leq t}|X_{s}^{(1)}|^{\bar{p}}\leq\frac{\lambda}{1-\lambda}\|\xi\|^{\bar{p}}+ \frac{1}{(1-\lambda)^{\bar{p}}}\sup_{0\leq s\leq t}|X_{s}^{(1)}-D(X_{s-\tau}^{(1)})|^{\bar{p}}.
\end{align}
Using the fractional It\^{o} formula in Lemma \ref{Itoformula}, it follows that
\begin{align*}
&|X_{t}^{(1)}-D(X_{t-\tau}^{(1)})|^{\bar{p}}\\
\leq &|\xi(0)-D(\xi(-\tau))|^{\bar{p}}\\
&+{\bar{p}}\int_{0}^t|X_{s}^{(1)}-D(X_{s-\tau}^{(1)})|^{{\bar{p}}-2}\langle X_{s}^{(1)}-D(X_{s-\tau}^{(1)}), b(X_{s}^{(1)},X_{s-\tau}^{(1)},\mathcal{L}_{X_{s}^{(0)}})\rangle{\rm d}s\\
&+\int_{0}^t {\bar{p}}({\bar{p}}-1)H(2H-1)|X_{s}^{(1)}-D(X_{s-\tau}^{(1)})|^{{\bar{p}}-2}\int_{0}^s(s-r)^{2H-2}\|\sigma(\mathcal{L}_{X_{r}^{(0)}})\|^2{\rm d}r{\rm d}s\\
&+{\bar{p}}\int_{0}^t|X_{s}^{(1)}-D(X_{s-\tau}^{(1)})|^{{\bar{p}}-1}\|\sigma(\mathcal{L}_{X_{s}^{(0)}})\|{\rm d}B_{s}^H.\\
=&|\xi(0)-D(\xi(-\tau))|^{\bar{p}}+Q_{1}^{(1)}(t)+Q_{2}^{(1)}(t)+Q_{3}^{(1)}(t).
\end{align*}
For any positive integer $N$, define the stopping time
\begin{align*}
\rho _{N}^{(1)}=T\wedge\inf\{t\in[0,T]:|X_{t}^{(1)}|\ge N\}.
\end{align*}
It yields that $\rho _{N}^{(1)}\to T$ a.s. as $N\to \infty$. Applying Assumptions \ref{zhonglixiangdetiaojian} and \ref{danbiantiaojian}, the Young's inequality, we see
\begin{align*}
  &\mathbb{E}\bigg(\sup_{0\leq s\leq t\wedge\rho_{N}^{(1)}}Q_{1}^{(1)}(s)\bigg)\\
  =&{\bar{p}}\mathbb{E}\int_{0}^{t\wedge\rho_{N}^{(1)}}|X_{s}^{(1)}-D(X_{s-\tau}^{(1)})|^{{\bar{p}}-2}\langle X_{s}^{(1)}-D(X_{s-\tau}^{(1)}), b(X_{s}^{(1)},X_{s-\tau}^{(1)},\mathcal{L}_{X_{s}^{(0)}})\rangle{\rm d}s\\
  \leq &C\mathbb{E}\int_{0}^{t\wedge\rho_{N}^{(1)}}|X_{s}^{(1)}-D(X_{s-\tau}^{(1)})|^{\bar{p}}{\rm d}s
  +C\mathbb{E}\int_{0}^{t\wedge\rho_{N}^{(1)}}\bigg[1+|X_{s}^{(1)}|^2+|X_{s-\tau}^{(1)}|^2+\mathbb{W}_2^2(\mathcal{L}_{X_{s}^{(0)}},\delta_{0})\bigg]^{{\bar{p}}/2}{\rm d}s\\
  \leq &C+C\int_{0}^{t}\mathbb{E}\bigg(\sup_{0\leq s\leq r\wedge\rho_{N}^{(1)}}|X_{s}^{(1)}|^{\bar{p}}\bigg){\rm d}r.
\end{align*}
Since Assumption \ref{sigmamanzudechushitiaojian} implies $\| \sigma(\mathcal{L}_{X_r^{(0)}}) \| \le C \big( 1 + \mathbb{W}_2(\mathcal{L}_{X_r^{(0)}}, \delta_0) \big) \le C$, this together with the Young's inequality, the Jensen inequality, and the Fubini theorem, it results in
\begin{align*}
&\mathbb{E}\bigg(\sup_{0\leq s\leq t\wedge\rho_{N}^{(1)}}Q_{2}^{(1)}(s)\bigg)\notag\\
=&\mathbb{E}\int_{0}^{t\wedge\rho_{N}^{(1)}} {\bar{p}}({\bar{p}}-1)H(2H-1)|X_{s}^{(1)}-D(X_{s-\tau}^{(1)})|^{{\bar{p}}-2}\int_{0}^s(s-r)^{2H-2}\|\sigma(\mathcal{L}_{X_{r}^{(0)}})\|^2{\rm d}r{\rm d}s\\ \notag
\leq &C\mathbb{E}\int_{0}^{t\wedge\rho_{N}^{(1)}}\int_{0}^s(s-r)^{2H-2}|X_{s}^{(1)}-D(X_{s-\tau}^{(1)})|^{{\bar{p}}-2}(1+\mathbb{W}_2(\mathcal{L}_{X_{r}^{(0)}},\delta_{0}))^2{\rm d}r{\rm d}s\\ \notag
\leq&C\mathbb{E}\int_{0}^{t\wedge\rho_{N}^{(1)}}|X_{s}^{(1)}-D(X_{s-\tau}^{(1)})|^{\bar{p}}\bigg(\int_{0}^s(s-r)^{2H-2}{\rm d}r\bigg){\rm d}s+C\mathbb{E}\int_{0}^{t\wedge\rho_{N}^{(1)}}\int_{0}^s(s-r)^{2H-2}{\rm d}r{\rm d}s\\ \notag
=&C\mathbb{E}\int_{0}^{t\wedge\rho_{N}^{(1)}}|X_{s}^{(1)}-D(X_{s-\tau}^{(1)})|^{\bar{p}}\bigg(\int_{0}^s(s-r)^{2H-2}{\rm d}r\bigg){\rm d}s+C\mathbb{E}\int_{0}^{t\wedge\rho_{N}^{(1)}}\int_{r}^t(s-r)^{2H-2}{\rm d}s{\rm d}r\\ \notag
= &C\mathbb{E}\int_{0}^{t\wedge\rho_{N}^{(1)}}\frac{s^{2H-1}}{2H-1}|X_{s}^{(1)}-D(X_{s-\tau}^{(1)})|^{\bar{p}}{\rm d}s+C\mathbb{E}\int_{0}^{t\wedge\rho_{N}^{(1)}}\frac{(t-r)^{2H-1}}{2H-1}{\rm d}r.\notag
\end{align*}
Combining with Assumption \ref{zhonglixiangdetiaojian} and the fact that $1/2< H< 1$, we obtain
\begin{align*}
\mathbb{E}\bigg(\sup_{0\leq s\leq t\wedge\rho_{N}^{(1)}}Q_{2}^{(1)}(s)\bigg)
\leq &C\int_{0}^t\mathbb{E}\bigg(\sup_{0\leq s\leq r\wedge\rho_{N}^{(1)}}|X_{s}^{(1)}|^{\bar{p}}\bigg){\rm d}r+C\mathbb{E}\int_{0}^{t\wedge\rho_{N}^{(1)}}(t-r)^{2H-1}{\rm d}r.
\end{align*}
For $Q_{3}^{(1)}(s)$, we use the fractional Burkholder-Davis-Gundy (BDG) inequality. For $H \in (1/2,1)$ and $p\ge 1$, there exists a constant $C_{H,p}>0$ such that
\begin{align}\label{fbdg}
&\mathbb{E}\left[ \sup_{0 \le s \le t}\left| \int_0^s u_r{\rm d}B_r^H \right|^p \right] \le C_{H,p} , \mathbb{E}\left[ \left( \int_0^t |u_r|^{1/H}{\rm d}r \right)^{pH} \right].
\end{align}
Applying this with $p=1$ and $u_r = |X_{r}^{(1)}-D(X_{r-\tau}^{(1)})|^{{\bar{p}}-1} \|\sigma(\mathcal{L}_{{X_{r}^{(0)}}})\|$, and noticing the fact that Assumption \ref{sigmamanzudechushitiaojian} yields $\| \sigma(\mathcal{L}_{X_r^{(0)}}) \| \le C \big( 1 + \mathbb{W}_2(\mathcal{L}_{X_r^{(0)}}, \delta_0) \big) \le C$, we get
\begin{align*}
&\mathbb{E}\left( \sup_{0 \le s \le t \wedge \rho_N^{(1)}} \left| Q_3^{(1)}(s) \right| \right)=\bar{p}\mathbb{E}\bigg(\sup_{0\leq s\leq t\wedge\rho_{N}^{(1)}}\left|\int_{0}^s u_r {\rm d}B_{r}^H\right|\bigg)
\leq   C\mathbb{E}\left[ \left( \int_0^{t\wedge\rho_N^{(1)}} |u_r|^{1/H}{\rm d}r \right)^{H} \right]\\
\leq & C\mathbb{E}\left[ \left( \int_0^{t\wedge\rho_N^{(1)}} |X_{r}^{(1)}-D(X_{r-\tau}^{(1)})|^{({\bar{p}}-1)/H}{\rm d}r \right)^{H} \right].
\end{align*}
For any $0 < \epsilon < 1$, the Young's inequality with exponents $\frac{\bar{p}}{\bar{p} - 1}$ and its conjugate gives
\begin{align*}
|X_r^{(1)} - D\big( X_{r-\tau}^{(1)} \big)|^{\bar{p} - 1} \le \epsilon |X_r^{(1)} - D\big( X_{r-\tau}^{(1)} \big)|^{\bar{p}} + C_\epsilon,
\end{align*}
where $C_\epsilon$ depends on $\epsilon, \bar{p}, H$. Integrating over $[0, t \wedge \rho_N^{(1)}]$, then
\begin{align*}
\int_0^{t \wedge \rho_N^{(1)}} |X_r^{(1)} - D\big( X_{r-\tau}^{(1)} \big)|^{(\bar{p} - 1)/H}  \mathrm{d}r
\le 2^{1/H-1}\epsilon^{1/H} \int_0^{t \wedge \rho_N^{(1)}} |X_r^{(1)} - D\big( X_{r-\tau}^{(1)} \big)|^{\bar{p}/H}  \mathrm{d}r + 2^{1/H-1}C_\epsilon^{1/H}t.
\end{align*}
Since $H \in (1/2, 1)$, we use the elementary inequality to get
\begin{align*}
\left( \int_0^{t \wedge \rho_N^{(1)}} |X_r^{(1)} - D\big( X_{r-\tau}^{(1)} \big)|^{(\bar{p} - 1)/H}  \mathrm{d}r \right)^H
\le& 2^{1-H}\epsilon \left( \int_0^{t \wedge \rho_N^{(1)}} |X_r^{(1)} - D\big( X_{r-\tau}^{(1)} \big)|^{\bar{p}/H}  \mathrm{d}r \right)^H +2^{1-H}C_\epsilon  t^H\\
\le&  2^{1-H}\epsilon t^H \cdot \sup_{0 \le r \le t \wedge \rho_N^{(1)}}|X_r^{(1)} - D\big( X_{r-\tau}^{(1)} \big)|^{\bar{p}}+2^{1-H}C_\epsilon  t^H.
\end{align*}
Taking expectations and using Assumption \ref{zhonglixiangdetiaojian}
\begin{align*}
\mathbb{E} \left( \sup_{0 \le r \le t \wedge \rho_N^{(1)}} |X_r^{(1)} - D\big( X_{r-\tau}^{(1)} \big)|^{\bar{p}} \right)
\le C \left( 1 + \mathbb{E} \left( \sup_{0 \le r \le t \wedge \rho_N^{(1)}} \left| X_r^{(1)} \right|^{\bar{p}} \right) \right).
\end{align*}
Combining all estimations, we arrive at
\begin{align}\label{Q3}
&\mathbb{E}\left( \sup_{0 \le s \le t \wedge \rho_N^{(1)}} \left| Q_3^{(1)}(s) \right| \right)
\le C \epsilon\mathbb{E} \left( \sup_{0 \le s \le t \wedge \rho_N^{(1)}} \left| X_s^{(1)} \right|^{\bar{p}} \right)  + C_\epsilon.
\end{align}
Choosing $\epsilon$ sufficiently small so that $C \epsilon< \frac{1}{2}$, and combining with the estimates for $Q_1^{(1)}(t)$ and $Q_2^{(1)}(t)$, we get
\begin{align*}
  \mathbb{E}\bigg(\sup_{0\leq s\leq t\wedge\rho_{N}^{(1)}}|X_{s}^{(1)}-D(X_{s-\tau}^{(1)})|^{\bar{p}}\bigg)
  \leq&C+C\int_{0}^t\mathbb{E}\bigg(\sup_{0\leq s\leq r\wedge\rho_{N}^{(1)}}|X_{s}^{(1)}|^{\bar{p}}\bigg){\rm d}r.
\end{align*}
By the Gronwall inequality together with \eqref{Xtdebudengshi}, we derive
\begin{align*}
  \mathbb{E}\bigg(\sup_{0\leq s\leq t\wedge\rho_{N}^{(1)}}|X_{s}^{(1)}|^{\bar{p}}\bigg)
  \leq &C.
\end{align*}
The Fatou lemma then leads to
\begin{align*}
\mathbb{E}\bigg(\sup_{0\leq s\leq t}|X_{s}^{(1)}|^{\bar{p}}\bigg)  \leq C.
\end{align*}
For $i=2$, we estimate the corresponding $Q_{1}^{(2)}(s), Q_{2}^{(2)}(s)$ and $Q_{3}^{(2)}(s)$ as follows
\begin{align*}
\mathbb{E}\bigg(\sup_{0\leq s\leq t\wedge\rho_{N}^{(2)}}Q_{1}^{(2)}(s)\bigg)
  \leq &C\bigg[\int_{0}^{t}\mathbb{E}\bigg(\sup_{0\leq s\leq r\wedge\rho_{N}^{(2)}}|X_{s}^{(2)}|^{\bar{p}}\bigg){\rm d}r+\int_{0}^{t}\mathbb{E}\bigg(\sup_{0\leq s\leq r\wedge\rho_{N}^{(2)}}|X_{s}^{(1)}|^{\bar{p}}\bigg){\rm d}r+1\bigg],
\end{align*}
\begin{align*}
\mathbb{E}\bigg(\sup_{0\leq s\leq t\wedge\rho_{N}^{(2)}}Q_{2}^{(2)}(s)\bigg)\leq &C\int_{0}^t\mathbb{E}\bigg(\sup_{0\leq s\leq r\wedge\rho_{N}^{(2)}}|X_{s}^{(2)}|^{\bar{p}}\bigg){\rm d}r+C\int_{0}^t(t-r)^{2H-1}\mathbb{E}\bigg(\sup_{0\leq s\leq r\wedge\rho_{N}^{(2)}}|X_{s}^{(1)}|^{\bar{p}}\bigg){\rm d}r,
\end{align*}
and
\begin{align*}
\mathbb{E}\bigg(\sup_{0\leq s\leq t\wedge\rho_{N}^{(2)}}Q_{3}^{(2)}(s)\bigg)\leq &C \epsilon\mathbb{E} \left( \sup_{0 \le s \le t \wedge \rho_N^{(2)}} \left| X_s^{(2)} \right|^{\bar{p}} \right)+C_{\epsilon}\int_{0}^{t}\mathbb{E}\bigg(\sup_{0\leq s\leq r\wedge\rho_{N}^{(2)}}|X_{s}^{(2)}|^{\bar{p}}\bigg){\rm d}r\\
&+C\int_{0}^t\mathbb{E}\bigg(\sup_{0\leq s\leq r\wedge\rho_{N}^{(2)}}|X_{s}^{(1)}|^{\bar{p}}\bigg){\rm d}r.
\end{align*}
Consequently, choosing $\epsilon$ sufficiently small so that $C \epsilon< \frac{1}{2}$ and using the Fatou lemma, the Gronwall inequality, it holds
\begin{align*}
\mathbb{E}\bigg(\sup_{0\leq s\leq t}|X_{s}^{(2)}|^{\bar{p}}\bigg)  \leq C.
\end{align*}
Finally, \eqref{pmoment} can be shown by induction.\\
{\bf Step 2. Existence} Define the difference processes $\Delta_{s}^{(k+1),(k)}=X^{(k+1)}_{s}-X_{s}^{(k)}$, $\Upsilon_{s}^{(k+1),(k)}=X^{(k+1)}_{s}-D(X_{s-\tau}^{(k+1)})-X_{s}^{(k)}+D(X_{s-\tau}^{(k)})$.
By Lemma \ref{Maodebudengshi}, we have
\begin{align*}
|\Delta_{s}^{(k+1),(k)}|^{\bar{p}}\leq\lambda|\Delta_{s-\tau}^{(k+1),(k)}|^{\bar{p}}+\frac{1}{(1-\lambda)^{\bar{p}-1}}|\Upsilon_{s}^{(k+1),(k)}|^{\bar{p}},
\end{align*}
Taking supremum and expectation gives
\begin{align}\label{zhong}
\mathbb{E}\left(\sup_{0 \leq s \leq t} |\Delta_{s}^{(k+1),(k)}|^{\bar{p}}\right) \leq \frac{1}{(1-\lambda)^{\bar{p}}} \mathbb{E}\left(\sup_{0 \leq s \leq t} |\Upsilon_{s}^{(k+1),(k)}|^{\bar{p}}\right).
\end{align}
For $t\in[0,T]$, employment of the fractional It\^{o} formula in Lemma \ref{Itoformula} and the elementary inequality provides
\begin{align*}
&\mathbb{E}\bigg(\sup_{0\leq s\leq t}|\Upsilon_{s}^{(k+1),(k)}|^{\bar{p}}\bigg)\\
\leq &{\bar{p}}\mathbb{E}\int_{0}^t|\Upsilon_{r}^{(k+1),(k)}|^{{\bar{p}}-2}\langle \Upsilon_{r}^{(k+1),(k)}, b(X^{(k+1)}_{r},X^{(k+1)}_{r-\tau},\mathcal{L}_{X^{(k)}_{r}})-b(X^{(k)}_{r},X^{(k)}_{r-\tau},\mathcal{L}_{X^{(k-1)}_{r}})\rangle{\rm d}r\\
&+\mathbb{E}\bigg(\sup_{0\leq s\leq t}\bigg\lvert\int_{0}^s {\bar{p}}({\bar{p}}-1)H(2H-1)|\Upsilon_{r}^{(k+1),(k)}|^{{\bar{p}}-2}\\
&\times\bigg(\int_{0}^r(r-u)^{2H-2}\|\sigma(\mathcal{L}_{X^{(k)}_{u}})-\sigma(\mathcal{L}_{X^{(k-1)}_{u}})\|^2{\rm d}u\bigg){\rm d}r\bigg)\\
&+{\bar{p}}\mathbb{E}\bigg(\sup_{0\leq s\leq t}\int_{0}^s|\Upsilon_{r}^{(k+1),(k)}|^{{\bar{p}}-1}\|\sigma(\mathcal{L}_{X^{(k)}_{r}})-\sigma(\mathcal{L}_{X^{(k-1)}_{r}})\|{\rm d}B_{r}^H\bigg).\\
:=&J_{1}(t)+J_{2}(t)+J_{3}(t).
\end{align*}
Applying Assumption \ref{danbiantiaojian} and the Young's inequality produces
\begin{align*}
  J_{1}(t)=&{\bar{p}}\mathbb{E}\int_{0}^t|\Upsilon_{r}^{(k+1),(k)}|^{{\bar{p}}-2}\langle \Upsilon_{r}^{(k+1),(k)}, b(X^{(k+1)}_{r},X^{(k+1)}_{r-\tau},\mathcal{L}_{X^{(k)}_{r}})-b(X^{(k)}_{r},X^{(k)}_{r-\tau},\mathcal{L}_{X^{(k-1)}_{r}})\rangle{\rm d}r\\
  \leq &C\mathbb{E}\int_{0}^t|X^{(k+1)}_{r}-D(X^{(k+1)}_{r-\tau})+X^{(k)}_{r}-D(X^{(k)}_{r-\tau})|^{\bar{p}}{\rm d}r\\
  &+C\mathbb{E}\int_{0}^t\bigg[|X^{(k+1)}_{r}-X^{(k)}_{r}|^2+|X^{(k+1)}_{r-\tau}-X^{(k)}_{r-\tau}|^2+\mathbb{W}_2^2(\mathcal{L}_{X^{(k)}_{r}},\mathcal{L}_{X^{(k-1)}_{r}})\bigg]^{{\bar{p}}/2}{\rm d}r\\
  \leq &C\int_{0}^t\mathbb{E}\bigg(\sup_{0\leq s\leq r}|\Delta_s^{(k+1),(k)}|^{\bar{p}}\bigg){\rm d}r+C\int_{0}^t\mathbb{E}\bigg(\sup_{0\leq s\leq r}|\Delta_s^{(k),(k-1)}|^{\bar{p}}\bigg){\rm d}r,
\end{align*}
where we used the fact that $\mathbb{W}_2^2(\mathcal{L}_{X^{(k)}_{r}},\mathcal{L}_{X^{(k-1)}_{r}}) \le \mathbb{E}^{\mathbb{P}}|X^{(k)}_{r}-X^{(k-1)}_{r}|^2$. Along with Assumption \ref{sigmamanzudechushitiaojian}, the Young's inequality, the H\"{o}lder inequality, and the Fubini theorem, it follows
\begin{align*}
J_{2}(t)=&\mathbb{E}\bigg(\sup_{0\leq s\leq t}\bigg\lvert\int_{0}^s {\bar{p}}({\bar{p}}-1)H(2H-1)|\Upsilon_{r}^{(k+1),(k)}|^{{\bar{p}}-2}\\
&\times\left(\int_{0}^r(r-u)^{2H-2}\|\sigma(\mathcal{L}_{X^{(k)}_{u}})-\sigma(\mathcal{L}_{X^{(k-1)}_{u}})\|^2{\rm d}u\bigg){\rm d}r\right)\\
\leq &C\int_{0}^t\int_{0}^r(r-u)^{2H-2}(\mathbb{E}|\Upsilon_{r}^{(k+1),(k)}|^{\bar{p}}+\mathbb{E}|X^{(k)}_{u}-X^{(k-1)}_{u}|^{\bar{p}}){\rm d}u{\rm d}r\\
\leq &C\int_{0}^t\mathbb{E}|\Upsilon_{r}^{(k+1),(k)}|^{\bar{p}}\bigg(\int_{0}^r(r-u)^{2H-2}{\rm d}u\bigg){\rm d}r+C\int_{0}^t\int_{u}^t(r-u)^{2H-2}\mathbb{E}|X^{(k)}_{u}-X^{(k-1)}_{u}|^{\bar{p}}{\rm d}r{\rm d}u\\
\leq &C\int_{0}^t\frac{r^{2H-1}}{2H-1}\mathbb{E}|\Upsilon_{r}^{(k+1),(k)}|^{\bar{p}}{\rm d}r+C\int_{0}^t\frac{(t-u)^{2H-1}}{2H-1}\mathbb{E}|X^{(k)}_{u}-X^{(k-1)}_{u}|^{\bar{p}}{\rm d}u.
\end{align*}
Since $1/2< H< 1$ and $t\in[0,T]$, we use Assumption \ref{zhonglixiangdetiaojian} to get
\begin{align*}
J_{2}(t)\leq &C\int_{0}^t\mathbb{E}\bigg(\sup_{0\leq s\leq r}|\Delta_s^{(k+1),(k)}|^{\bar{p}}\bigg){\rm d}r+C\int_{0}^t(t-r)^{2H-1}\mathbb{E}\bigg(\sup_{0\leq s\leq r}|\Delta_s^{(k),(k-1)}|^{\bar{p}}\bigg){\rm d}r.
\end{align*}
In the similar way as dealing with $Q_3^{(1)}(t)$ to get \eqref{Q3}, an application of fractional BDG inequality \eqref{fbdg} with $p=1$ yields
\begin{align*}
J_{3}(t)= & {\bar{p}} \mathbb{E}\left[ \sup_{0\leq s\leq t} \left| \int_{0}^s |\Upsilon_{r}^{(k+1),(k)}|^{{\bar{p}}-1} \|\sigma(\mathcal{L}_{X^{(k)}_{r}}) - \sigma(\mathcal{L}_{X^{(k-1)}_{r}})\| {\rm d}B_{r}^H \right| \right]\\
 \le &C \mathbb{E}\left[ \left( \int_0^t \left( |\Upsilon_{r}^{(k+1),(k)}|^{{\bar{p}}-1} \|\sigma(\mathcal{L}_{X^{(k)}_{r}}) - \sigma(\mathcal{L}_{X^{(k-1)}_{r}})\| \right)^{1/H} {\rm d}r \right)^{H} \right]
\end{align*}
For each $r \in [0,t]$, by the Young's inequality, there exists $0<\epsilon<1$ such that
\begin{align*}
&|\Upsilon_{r}^{(k+1),(k)}|^{{\bar{p}}-1} \|\sigma(\mathcal{L}_{X^{(k)}_{r}}) - \sigma(\mathcal{L}_{X^{(k-1)}_{r}})\|\leq \epsilon |\Upsilon_{r}^{(k+1),(k)}|^{{\bar{p}}} + C_\epsilon \|\sigma(\mathcal{L}_{X^{(k)}_{r}}) - \sigma(\mathcal{L}_{X^{(k-1)}_{r}})\|^{{\bar{p}}}.
\end{align*}
Assumption \ref{sigmamanzudechushitiaojian} implies that
\begin{align*}
\|\sigma(\mathcal{L}_{X^{(k)}_{r}}) - \sigma(\mathcal{L}_{X^{(k-1)}_{r}})\| \leq K_4 \mathbb{W}_2(\mathcal{L}_{X^{(k)}_{r}}, \mathcal{L}_{X^{(k-1)}_{r}}) \leq K_4 (\mathbb{E}|\Delta_r^{(k),(k-1)}|^2)^{1/2}.
\end{align*}
Since ${\bar{p}} \geq 2$, using the fact that $(\mathbb{E}|\Upsilon_r^{(k),(k-1)}|^2)^{{\bar{p}}/2} \leq \mathbb{E}|\Upsilon_r^{(k),(k-1)}|^{{\bar{p}}}$, then
\begin{align*}
&|\Upsilon_{r}^{(k+1),(k)}|^{{\bar{p}}-1} \|\sigma(\mathcal{L}_{X^{(k)}_{r}}) - \sigma(\mathcal{L}_{X^{(k-1)}_{r}})\|
\leq \epsilon |\Upsilon_{r}^{(k+1),(k)}|^{{\bar{p}}} + C_\epsilon (\mathbb{E}|\Delta_r^{(k),(k-1)}|^2)^{{\bar{p}}/2}\\
\leq& \epsilon |\Upsilon_{r}^{(k+1),(k)}|^{{\bar{p}}} + C_\epsilon \mathbb{E}|\Delta_r^{(k),(k-1)}|^{{\bar{p}}}.
\end{align*}
Noting that for $\frac{1}{H}>1$, we can apply the elementary inequality to get
\begin{align*}
&\left( \int_0^t \left( \epsilon |\Upsilon_{r}^{(k+1),(k)}|^{{\bar{p}}} + C_\epsilon \mathbb{E}|\Delta_r^{(k),(k-1)}|^{{\bar{p}}} \right)^{1/H} {\rm d}r \right)^{H} \\
\leq& 2^{1-H}\epsilon \left( \int_0^t |\Upsilon_{r}^{(k+1),(k)}|^{{\bar{p}}/H} {\rm d}r \right)^{H} + C_\epsilon  2^{1-H}\left( \int_0^t \left( \mathbb{E}|\Delta_r^{(k),(k-1)}|^{{\bar{p}}} \right)^{1/H} {\rm d}r \right)^{H}.
\end{align*}
Taking expectation, it holds
\begin{align*}
&\mathbb{E}\left[\left( \int_0^t |\Upsilon_{r}^{(k+1),(k)}|^{{\bar{p}}/H} {\rm d}r \right)^{H} \right]
\leq t^H \mathbb{E}\left( \sup_{0\leq r \leq t} |\Upsilon_{r}^{(k+1),(k)}|^{{\bar{p}}} \right),
\end{align*}
and
\begin{align*}
&\mathbb{E}\left[ \left( \int_0^t \left( \mathbb{E}|\Delta_r^{(k),(k-1)}|^{{\bar{p}}} \right)^{1/H} {\rm d}r \right)^{H} \right]
\leq Ct^H \int_0^t \mathbb{E}|\Delta_r^{(k),(k-1)}|^{{\bar{p}}} {\rm d}r.
\end{align*}
Combining the estimates together, we conclude
\begin{align}\label{J3}
J_{3}(t)\leq & C\epsilon \mathbb{E}\left(\sup_{0\leq s \leq t} |\Upsilon_{s}^{(k+1),(k)}|^{{\bar{p}}} \right)+ C_\epsilon \int_0^t \mathbb{E}\left( \sup_{0\leq s \leq r} |\Delta_s^{(k),(k-1)}|^{{\bar{p}}} \right){\rm d}r.
\end{align}
Finally, the estimation of $J_3(t)$ combined with the estimates for the drift terms $J_1(t)$ and $J_2(t)$, we obtain
\begin{align*}
&\mathbb{E}\left(\sup_{0\leq s \leq t} |\Upsilon_s^{(k+1),(k)}|^{\bar{p}} \right)\\
\leq & C\epsilon \mathbb{E}\left( \sup_{0\leq s \leq t} |\Upsilon_s^{(k+1),(k)}|^{\bar{p}} \right)+C\int_{0}^t\mathbb{E}\bigg(\sup_{0\leq s\leq r}|\Delta_s^{(k+1),(k)}|^{\bar{p}}\bigg){\rm d}r\\
&+C\int_{0}^t\mathbb{E}\bigg(\sup_{0\leq s\leq r}|\Delta_s^{(k),(k-1)}|^{\bar{p}}\bigg){\rm d}r+C\int_{0}^t(t-r)^{2H-1}\mathbb{E}\bigg(\sup_{0\leq s\leq r}|\Delta_s^{(k),(k-1)}|^{\bar{p}}\bigg){\rm d}r.
\end{align*}
Choosing $\epsilon$ small enough so that $C\epsilon < \frac{1}{2}$ and absorbing the first term, it implies
\begin{align*}
\mathbb{E}\left(\sup_{0\leq s \leq t} |\Upsilon_s^{(k+1),(k)}|^{\bar{p}} \right)\le & C\int_{0}^t\mathbb{E}\bigg(\sup_{0\leq s\leq r}|\Delta_s^{(k+1),(k)}|^{\bar{p}}\bigg){\rm d}r+C\int_{0}^t\mathbb{E}\bigg(\sup_{0\leq s\leq r}|\Delta_s^{(k),(k-1)}|^{\bar{p}}\bigg){\rm d}r\\
&+C\int_{0}^t(t-r)^{2H-1}\mathbb{E}\bigg(\sup_{0\leq s\leq r}|\Delta_s^{(k),(k-1)}|^{\bar{p}}\bigg){\rm d}r.
\end{align*}
Since singular kernels $\int_{0}^t(t-r)^{2H-1}\mathbb{E}\bigg(\sup_{0\leq s\leq r}|\Delta_s^{(k),(k-1)}|^{\bar{p}}\bigg){\rm d}r$ are typically ``stronger" than regular kernels $\int_{0}^t(t-r)^{0}\mathbb{E}\bigg(\sup_{0\leq s\leq r}|\Delta_s^{(k),(k-1)}|^{\bar{p}}\bigg){\rm d}r$, we derive from \eqref{zhong} that
\begin{align*}
&\mathbb{E}\bigg(\sup_{0\le s\le t}|\Delta_s^{(k+1),(k)}|^{\bar{p}}\bigg)\\
\leq&C\int_{0}^t\mathbb{E}\bigg(\sup_{0\leq s\leq r}|\Delta_s^{(k+1),(k)}|^{\bar{p}}\bigg){\rm d}r+C\int_{0}^t(t-r)^{2H-1}\mathbb{E}\bigg(\sup_{0\leq s\leq r}|\Delta_s^{(k),(k-1)}|^{\bar{p}}\bigg){\rm d}r.
\end{align*}
Using the Gronwall inequality leads to
\begin{align*}
\mathbb{E}\bigg(\sup_{0\le s\le t}|\Delta_s^{(k+1),(k)}|^{\bar{p}}\bigg)
\leq Ce^{Ct}\int_{0}^t(t-r)^{2H-1}\mathbb{E}\bigg(\sup_{0\leq s\leq r}|\Delta_s^{(k),(k-1)}|^{\bar{p}}\bigg){\rm d}r.
\end{align*}
By induction, we arrive at
\begin{align*}
\mathbb{E}\bigg(\sup_{0\le s\le t}|\Delta_s^{(k+1),(k)}|^{\bar{p}}\bigg)
\leq\frac{(Ce^{Ct}\Gamma(2H)t^{2H})^{k}}{\Gamma(2kH+1)}\mathbb{E}\bigg(\sup_{0\le s\le t}|\Delta_s^{(1),(0)}|^{\bar{p}}\bigg).
\end{align*}
With the properties of the Gamma function and the fact that $\mathbb{E}\bigg(\sup\limits_{0\le s\le t}|\Delta_s^{(1),(0)}|^{\bar{p}}\bigg)\le C$ derived by \eqref{pmoment}, we can show that $\{X_{t}^{(k)}\}$ is a Cauchy sequence, then the limit, denoted by $X$, is a solution of \eqref{MVdrivenbyfBm}.\\
{\bf Step 3. Uniqueness} Let $X_{t}$ and $Y_{t}$ be two solutions of \eqref{MVdrivenbyfBm} with $X_{t}=Y_{t}=\xi,\ \ t\in[-\tau,0]$.
By Lemma \ref{Maodebudengshi} again, we establish
\begin{align*}
\sup_{0\leq r\leq t}|X_{r}-Y_{r}|^{\bar{p}}\leq \frac{1}{(1-\lambda)^{\bar{p}}}\sup_{0\leq r\leq t}|\Upsilon_{r}|^{\bar{p}},
\end{align*}
where $\Upsilon_{r}=X_{r}-D(X_{r-\tau})-Y_{r}+D(Y_{r-\tau})$. Similar to the process of Step 2, it holds that
\begin{align*}
\mathbb{E}\bigg(\sup_{0\le s\le t}|X_{s}-Y_{s}|^{\bar{p}}\bigg)
\leq&C\int_{0}^t\mathbb{E}\bigg(\sup_{0\leq s\leq r}|X_{s}-Y_{s}|^{\bar{p}}\bigg){\rm d}r
+C\int_{0}^t(t-r)^{2H-1}\mathbb{E}\bigg(\sup_{0\leq s\leq r}|X_{s}-Y_{s}|^{\bar{p}}\bigg){\rm d}r.
\end{align*}
Further, we derive
\begin{align*}
\mathbb{E}\bigg(\sup_{0\le s\le t}|X_{s}-Y_{s}|^{\bar{p}}\bigg)
\leq C\int_{0}^t(t-r)^{2H-1}\mathbb{E}\bigg(\sup_{0\leq s\leq r}|X_{s}-Y_{s}|^{\bar{p}}\bigg){\rm d}r.
\end{align*}
Although the kernel $(t-r)^{2H-1}$ is singular as $2H-1>0$, it is integrable on [0, T]. Therefore, applying Lemma \ref{Hlemma} with $a=0$, $\beta=2H$ and $\gamma=1$, we find that the generalized Mittag-Leffler function is bounded, guaranteeing a unique zero solution for the inequality. That is,
 \begin{align*}
 \mathbb{E}\bigg(\sup_{0\le s\le t}|X_{s}-Y_{s}|^{\bar{p}}\bigg)=0.
 \end{align*}
Therefore, $X_t = Y_t$ almost surely for all $t \in [0,T]$. This completes the proof.
\end{proof}

\begin{rem}
{\rm For the existence and uniqueness of the McKean-Vlasov NSDDEs, we adopt the Picard approximation method to construct a Cauchy sequence and show the existence and uniqueness of the solution. Since the super-linearly growing coefficients in our proposed model introduces additional technical complexities that demand refined analytical tools, our analysis differs fundamentally from those of related works. For McKean-Vlasov SDEs driven by standard Brownian motions (e.g., \cite{SXW22}) or by fractional Brownian motions without super-linearly growing coefficients (e.g., \cite{WXT25, GGL25}), the uniqueness is typically proven via direct estimates of the difference between two candidate solutions. In contrast, our model involves fractional noise and super-linear coefficients, which necessitates the use of the fractional It\^{o} formula to properly account for the non-Markovian and rough nature of the fractional Brownian motions. This application of the fractional It\^{o} formula not only resolves the technical hurdles posed by the fractional noise but also provides a generalizable framework for analyzing McKean-Vlasov SDEs driven by fractional Brownian motions.}
\end{rem}

\section{Propagation of Chaos}
Since the coefficients of \eqref{MVdrivenbyfBm} depend on the distributions of state variables, we use stochastic interacting particle systems to approximate it. Firstly, considering a non-interacting particle system given by
\begin{align}\label{noninteractingparticlesystem}
{\rm d}(X_{t}^i-D(X_{t-\tau}^i))=b(X_{t}^i,X_{t-\tau}^i,\mathcal{L}_{X_{t}^i}){\rm d}t+\sigma(\mathcal{L}_{X_{t}^i}){\rm d}B_{t}^{H,i},\ \ X_{0}^i=Y_{0}^i=\xi,\ \ i=1,2,\cdots,N,
\end{align}
here, $|\xi_{t}^i-\xi_{s}^i|^p\leq |t-s|^p$. To deal with $\mathcal{L}_{X_{t}^i}$, we exploit the following interacting particle system of the form
\begin{align}\label{interactingparticlesystem}
{\rm d}(X_{t}^{i,N}-D(X_{t-\tau}^{i,N}))=b(X_{t}^{i,N},X_{t-\tau}^{i,N},\mu_{t}^{X,N}){\rm d}t+\sigma(\mu_{t}^{X,N}){\rm d}B_{t}^{H,i},\ \ i=1,2,\cdots,N,
\end{align}
where the empirical measures is defined by
\begin{align*}
\mu^{X,N}_{t}=\frac{1}{N}\sum_{i=1}^N\delta_{X_{t}^{i,N}},
\end{align*}
here, $\delta_{x}$ denotes the Dirac measure at point $x$. In order to prove the desired result, we give the following Lemma.

\begin{lem}\cite[Lemma 3.2]{HGZ24}\label{measuredistance}
{\rm For two empirical measures $\mu_{t}^N=\frac{1}{N}\sum\limits_{i=1}\limits^{N}\delta_{Z^{1,i,N}_{t}}$ and $\nu_{t}^N=\frac{1}{N}\sum\limits_{i=1}\limits^{N}\delta_{Z_{t}^{2,i,N}}$, then}
\begin{align*}
\mathbb{E}\mathbb{W}_{2}(\mu_{t}^N,\nu_{t}^N)\leq \mathbb{E}\bigg(\frac{1}{N}\sum_{i=1}^N\bigg\lvert Z_{t}^{1,i,N}-Z_{t}^{2,i,N}\bigg\rvert^{2}\bigg)^{1/2}.
\end{align*}
\end{lem}
Next, we consider the theory of the propagation of chaos.

\begin{lem}\label{propagationyoujie}
{\rm Let Assumptions \ref{chushizhixidetiaojian}-\ref{sigmamanzudechushitiaojian} hold. For $\bar{p}\ge2$, we have
\begin{align*}
\mathbb{E}\bigg(\sup_{t\in[-\tau,T]}|X_{t}^{i}|^{\bar{p}}\bigg)+\mathbb{E}\bigg(\sup_{t\in[-\tau,T]}|X_{t}^{i,N}|^{\bar{p}}\bigg)\leq C,\ \ i=1,2,\cdots,N,
\end{align*}
where $C$ is a positive constant depending on ${\bar{p}},T,L,H, \xi$ but independent of $N$.}
\end{lem}

\begin{proof}
We will proceed the proof in the similar way as Theorem \ref{Xtnyoujie}. Firstly, Assumption \ref{zhonglixiangdetiaojian} and Lemma \ref{Maodebudengshi} imply that
\begin{align}\label{xsin}
\sup_{0\leq s\leq t}|X_{s}^{i,N}|^{\bar{p}}\leq\frac{\lambda}{1-\lambda}\parallel\xi\parallel^{\bar{p}}+ \frac{1}{(1-\lambda)^{{\bar{p}}}}\sup_{0\leq s\leq t}|X_{s}^{i,N}-D(X_{s-\tau}^{i,N})|^{\bar{p}}.
\end{align}
We now concentrate on the estimation of the second term. The application of the fractional It\^{o} formula leads to
\begin{align*}
&|X_{t}^{i,N}-D(X_{t-\tau}^{i,N})|^{\bar{p}}\\
\leq& |\xi(0)-D(\xi(-\tau))|^{\bar{p}}\\
&+\bar{p}\int_{0}^t|X_{s}^{i,N}-D(X_{s-\tau}^{i,N})|^{{\bar{p}}-2}\big\langle X_{s}^{i,N}-D(X_{s-\tau}),b\big(X_{s}^{i,N},X_{s-\tau}^{i,N},\mu_{s}^{X,N}\big)\big\rangle{\rm d}s\\
&+{\bar{p}}\int_{0}^t|X_{s}^{i,N}-D(X_{s-\tau}^{i,N})|^{\bar{p}-1}\|\sigma\big(\mu_{s}^{X,N}\big)\|{\rm d}B_{s}^H\\
&+{\bar{p}}({\bar{p}}-1)H(2H-1)\int_{0}^t|X_{s}^{i,N}-D(X_{s-\tau}^{i,N})|^{{\bar{p}}-2}\bigg(\int_{0}^s(s-r)^{2H-2}\|\sigma(\mu_{r}^{X,N})\|^2{\rm d}r\bigg){\rm d}s\\
=:&|\xi(0)-D(\xi(-\tau))|^{\bar{p}}+I_{1}(t)+I_{2}(t)+I_{3}(t).
\end{align*}
By means of Assumptions \ref{zhonglixiangdetiaojian}-\ref{danbiantiaojian} and the Young's inequality, we conclude
\begin{align*}
&\mathbb{E}\bigg(\sup_{0\leq s\leq t}I_{1}(s)\bigg)\\
= &\mathbb{E}\int_{0}^{t} {\bar{p}}|X_{s}^{i,N}-D(X_{s-\tau}^{i,N})|^{{\bar{p}}-2}\big\langle X_{s}^{i,N}-D(X_{s-\tau}^{i,N}),b\big(X_{s}^{i,N},X_{s-\tau}^{i,N},\mu_{s}^{X,N}\big)\big\rangle{\rm d}s\\
\leq&C\mathbb{E}\int_{0}^{t}[|X_{s}^{i,N}-D(X_{s-\tau}^{i,N})|^{\bar{p}}+(1+|X_{s}^{i,N}|^2+|X_{s-\tau}^{i,N}|^2+\mathbb{W}_{2}^2(\mu_{s}^{X,N},\delta_{0}))^{{\bar{p}}/2}]{\rm d}s\\
\leq &C\mathbb{E}\int_{0}^{t}[(1+|X_{s}^{i,N}|^{\bar{p}}+|X_{s-\tau}^{i,N}|^{\bar{p}}+\mathbb{W}_{2}^{\bar{p}}(\mu_{s}^{X,N},\delta_{0})]{\rm d}s.
\end{align*}
Utilizing Assumption \ref{sigmamanzudechushitiaojian}, the H\"{o}lder inequality, the Young's inequality and the similar way as achieving \eqref{J3}, it produces
\begin{align*}
\mathbb{E}\bigg(\sup_{0\leq s\leq t}I_{2}(s)\bigg)=&{\bar{p}}\mathbb{E}\int_{0}^t|X_{s}^{i,N}-D(X_{s-\tau}^{i,N})|^{\bar{p}-1}\|\sigma\big(\mu_{s}^{X,N}\big)\|{\rm d}B_{s}^H\\
\leq &C\epsilon \mathbb{E}\left(\sup_{0\leq s \leq t} |X_{s}^{i,N}|^{{\bar{p}}} \right)+ C_\epsilon \mathbb{E}\int_{0}^{t}|X_{s-\tau}^{i,N}|^{\bar{p}}{\rm d}s+C\mathbb{E}\int_{0}^{t}\mathbb{W}_{2}^{\bar{p}}(\mu_{s}^{X,N},\delta_{0}){\rm d}s+C,
\end{align*}
where $0<\epsilon<1$. Again by Assumption \ref{sigmamanzudechushitiaojian}, the Young's inequality and the Fubini theorem, it results in
\begin{align*}
&\mathbb{E}\bigg(\sup_{0\leq s\leq t}I_{3}(s)\bigg)\\
=&{\bar{p}}({\bar{p}}-1)H(2H-1)\mathbb{E}\int_{0}^t|X_{s}^{i,N}-D(X_{s-\tau}^{i,N})|^{{\bar{p}}-2}\bigg(\int_{0}^s(s-r)^{2H-2}\|\sigma(\mu_{r}^{X,N})\|^2{\rm d}r\bigg){\rm d}s\\
\leq &C\mathbb{E}\int_{0}^{t}|X_{s}^{i,N}|^{\bar{p}}{\rm d}s+C\mathbb{E}\int_{0}^{t}|X_{s-\tau}^{i,N}|^{\bar{p}}{\rm d}s+C\int_{0}^t(t-s)^{2H-1}\mathbb{E}|X_{s}^{i,N}|^{\bar{p}}{\rm d}s,
\end{align*}
where we have used the fact that
\begin{align*}
 \mathbb{W}_{2}(\mu_{s}^{X,N},\delta_{0})=\mathcal{W}_{2}(\mu_{s}^{X,N})=\bigg(\frac{1}{N}\sum_{i=1}^N|X_{s}^{i,N}|^{2}\bigg)^{1/2}.
\end{align*}
Since all $X_{s}^{i,N}$ are identically distributed, the Minkowski inequality and Lemma \ref{measuredistance} lead to
\begin{align*}
&\mathbb{E}\bigg(\frac{1}{N}\sum_{i=1}^N|X_{s}^{i,N}|^{2}\bigg)^{{\bar{p}}/2}
\leq \bigg(\frac{1}{N}\sum_{i=1}^N\parallel |X_{s}^{i,N}|^{2}\parallel_{L^{\bar{p}/2}}\bigg)^{\bar{p}/2}
=\bigg(\frac{1}{N} \bigg(\sum_{i=1}^N(\mathbb{E}|X_{s}^{i,N}|^{\bar{p}})^{2/\bar{p}}\bigg)^{\bar{p}/2}=\mathbb{E}|X_{s}^{i,N}|^{\bar{p}}.
\end{align*}
To sum up, we obtain
\begin{align*}
  \mathbb{E}\bigg(\sup_{0\leq s\leq t}|X_{s}^{i,N}-D(X_{s-\tau}^{i,N})|^{\bar{p}}\bigg)
  \leq&C+C\epsilon \mathbb{E}\left(\sup_{0\leq s \leq t} |X_{s}^{i,N}|^{{\bar{p}}} \right)+C\int_{0}^t\mathbb{E}\bigg(\sup_{0\leq s\leq r}|X_{s}^{i,N}|^{\bar{p}}\bigg){\rm d}r\\
  &+C\int_{0}^t(t-r)^{2H-1}\mathbb{E}\bigg(\sup_{0\leq s\leq r}|X_{s}^{i,N}|^{\bar{p}}\bigg){\rm d}r.
\end{align*}
Consequently, choosing some $\epsilon$ small enough to ensure $C\epsilon < \frac{1}{2}$ and using \eqref{xsin} to get
\begin{align}\label{sss}
  \mathbb{E}\bigg(\sup_{0\leq s\leq t}|X_{s}^{i,N}|^{\bar{p}}\bigg)
  \leq &C+C\int_{0}^t\mathbb{E}\bigg(\sup_{0\leq s\leq r}|X_{s}^{i,N}|^{\bar{p}}\bigg){\rm d}r+C\int_{0}^t(t-r)^{2H-1}\mathbb{E}\bigg(\sup_{0\leq s\leq r}|X_{s}^{i,N}|^{\bar{p}}\bigg){\rm d}r.
\end{align}
Since the singular kernel $(t-s)^{2H-1}$ is integrable on the finite interval $[0,T]$, we can bound
\begin{align*}
\int_{0}^t(t-s)^{2H-1}{\rm d}s\leq C.
\end{align*}
Rewriting \eqref{sss} by ignoring the regular kernel as
\begin{align*}
\mathbb{E}\bigg(\sup_{0\leq s\leq t}|X_{s}^{i,N}|^{\bar{p}}\bigg)\leq C+C\int_{0}^t(t-r)^{2H-1}\mathbb{E}\bigg(\sup_{0\leq s\leq r}|X_{s}^{i,N}|^{\bar{p}}\bigg){\rm d}r,
\end{align*}
then, Lemma \ref{Hlemma} implies
\begin{align*}
\mathbb{E}\bigg(\sup_{0\leq s\leq t}|X_{s}^{i,N}|^{\bar{p}}\bigg)\leq CE_{2H,1}((\Gamma(2H))^{1/2H}t).
\end{align*}
The Mittag-Leffler function, defined by the power series $E_{2H,1}((\Gamma(2H))^{1/2H}t)$, is an entire function. It is a well-established result that this series converges uniformly on any compact subset of the complex plane, and therefore on any finite interval $[0,T]$ (see, e.g.,\cite{HMS11}), which means $\mathbb{E}\bigg(\sup_{t\in[-\tau,T]}|X_{t}^{i,N}|^{\bar{p}}\bigg)$ is bounded. Finally, the boundness of $\mathbb{E}\bigg(\sup_{t\in[-\tau,T]}|X_{t}^{i}|^{\bar{p}}\bigg)$ can be shown similarly. This completes the proof.
\end{proof}

\begin{lem}\label{propagationofchaosholdfinitemoment}
{\rm Let Assumptions \ref{chushizhixidetiaojian}-\ref{sigmamanzudechushitiaojian} hold. Fix any $\bar{p}>4$ in Lemma \ref{propagationyoujie}. Then, for $p\in[2,\frac{\bar{p}}{2})$, we have}
\begin{align*}
 \mathbb{E}\bigg(\sup_{t\in[0,T]}|X_{t}^i-X_{t}^{i,N}|^p\bigg)\leq C\times\left\{\begin{array}{ll}
  N^{-\frac{1}{2}},\ \ p>\frac{d}{2},\\
     N^{-\frac{1}{2}}\log(1+N),\ p=\frac{d}{2},\\
     N^{-\frac{p}{d}},\ p\in[2,\frac{d}{2}),
\end{array}
\right.
\end{align*}
{\rm where $C$ is a positive constant depending on $p,T,H,L$ but independent of $N$ and $d$.}
\end{lem}

\begin{proof}
For $i=1,2,\cdots,N$, by \eqref{noninteractingparticlesystem} and \eqref{interactingparticlesystem}, set
\begin{align*}
\Upsilon_{t}^i=X_{t}^i-D(X_{t-\tau}^i)-X_{t}^{i,N}+D(X_{t-\tau}^{i,N}).
\end{align*}
By using the fractional It\^{o} formula to $|\Upsilon_{t}^i|^p$, we can show
\begin{align*}
 &|\Upsilon_{t}^i|^p-|\Upsilon_{0}^i|^p\\
 \leq &p\int_{0}^t|\Upsilon_{s}^i|^{p-2}\langle \Upsilon_{s}^i,b(X_{s}^i,X_{s-\tau}^i,\mathcal{L}_{X_{s}^i})-b(X_{s}^{i,N},X_{s-\tau}^{i,N},\mu_{s}^{X,N})\rangle{\rm d}s\\
 &+p\int_{0}^t|\Upsilon_{s}^i|^{p-2}\langle\Upsilon_{s}^i,\sigma(\mathcal{L}_{X_{s}^i})-\sigma(\mu_{s}^{X,N})\rangle{\rm d}B_{s}^{H,i}\\
 &+p(p-1)H(2H-1)\int_{0}^t|\Upsilon_{s}^i|^{p-2}\bigg(\int_{0}^s(s-r)^{2H-2}\|\sigma(\mathcal{L}_{X_{r}^i})-\sigma(\mu_{r}^{X,N})\|^2{\rm d}r\bigg){\rm d}s\\
 =:&H_{1}(t)+H_{2}(t)+H_{3}(t).
\end{align*}
Assumptions \ref{zhonglixiangdetiaojian}, \ref{danbiantiaojian} and the Young's inequality lead to
\begin{align*}
 &\mathbb{E}\bigg(\sup_{0\leq s\leq t}H_{1}(s)\bigg)
 = p\mathbb{E}\int_{0}^t|\Upsilon_{s}^i|^{p-2}\langle \Upsilon_{s}^i,b(X_{s}^i,X_{s-\tau}^i,\mathcal{L}_{X_{s}^i})-b(X_{s}^{i,N},X_{s-\tau}^{i,N},\mu_{s}^{X,N})\rangle{\rm d}s\\
 \leq &C\mathbb{E}\int_{0}^t|X_{t}^i-X_{t}^{i,N}|^{p}{\rm d}s+C\mathbb{E}\int_{0}^t|X_{s-\tau}^i-X_{s-\tau}^{i,N}|^p{\rm d}s+C\mathbb{E}\int_{0}^t\mathbb{W}_{2}^p(\mathcal{L}_{X_{s}^i},\mu_{s}^{X,N}){\rm d}s.
\end{align*}
As the same way of achieving \eqref{J3}, we use Assumption \ref{sigmamanzudechushitiaojian}, the Young's inequality and the H\"{o}lder inequality to get
\begin{align*}
\mathbb{E}\bigg(\sup_{0\leq s\leq t}H_{2}(s)\bigg)
=&p\mathbb{E}\int_{0}^t|\Upsilon_{s}^i|^{p-2}\langle\Upsilon_{s}^i,\sigma(\mathcal{L}_{X_{s}^i})-\sigma(\mu_{s}^{X,N})\rangle{\rm d}B_{s}^{H,i}\\
\leq &C\epsilon \mathbb{E}\left(\sup_{0\leq s \leq t} |\Upsilon_{s}^i|^{{p}} \right)+C\mathbb{E}\int_{0}^t\mathbb{W}_{2}^p(\mathcal{L}_{X_{s}^i},\mu_{s}^{X,N}){\rm d}s.
\end{align*}
Further, Assumptions \ref{zhonglixiangdetiaojian}, \ref{sigmamanzudechushitiaojian}, Remark \ref{Itoformula}, the Young's inequality and Fubini theorem give
\begin{align*}
&\mathbb{E}\bigg(\sup_{0\leq s\leq t}H_{3}(s)\bigg)\\
=&p(p-1)H(2H-1)\int_{0}^t|\Upsilon_{s}^i|^{p-2}\bigg(\int_{0}^s(s-r)^{2H-2}\|\sigma(\mathcal{L}_{X_{r}^i})-\sigma(\mu_{r}^{X,N})\|^2{\rm d}r\bigg){\rm d}s\\
\leq &C\mathbb{E}\int_{0}^{t}|\Upsilon_{s}^i|^{p}{\rm d}s+C\mathbb{E}\int_{0}^t(t-s)^{2H-1}\mathbb{W}_{2}^{p}\big(\mathcal{L}_{X_{s}^i},\mu_{s}^{X,N}\big){\rm d}s\\
\le&C\mathbb{E}\int_{0}^t|X_{t}^i-X_{t}^{i,N}|^{p}{\rm d}s+C\mathbb{E}\int_{0}^t|X_{s-\tau}^i-X_{s-\tau}^{i,N}|^p{\rm d}s+C\mathbb{E}\int_{0}^t(t-s)^{2H-1}\mathbb{W}_{2}^p(\mathcal{L}_{X_{s}^i},\mu_{s}^{X,N}){\rm d}s.
\end{align*}
Noting that
\begin{align*}
\mathbb{W}_{2}(\mathcal{L}_{X_{s}^i},\mu_{s}^{X,N})\leq \mathbb{W}_{2}(\mathcal{L}_{X_{s}^i},\mu_{s}^{X})+\mathbb{W}_{2}(\mu_{s}^{X},\mu_{s}^{X,N}),
\end{align*}
then, for $p\ge2$, Lemma \ref{measuredistance} leads to
\begin{align*}
 \mathbb{E}\mathbb{W}_{2}^p(\mathcal{L}_{X_{s}^i},\mu_{s}^{X,N})= &\mathbb{E}[\mathbb{W}_{2}^{2}(\mathcal{L}_{X_{s}^i},\mu_{s}^{X,N})]^{p/{2}}\\
\leq &\mathbb{E}\bigg(2\mathbb{W}_{2}^{2}(\mathcal{L}_{X_{s}^i},\mu_{s}^{X})+2\mathbb{W}_{2}^{2}(\mu_{s}^{X},\mu_{s}^{X,N})\bigg)^{p/2}\\
\leq &C\mathbb{E}\mathbb{W}_{2}^p(\mathcal{L}_{X_{s}^i},\mu_{s}^{X})+C\mathbb{E}\mathbb{W}_{2}^p(\mu_{s}^{X},\mu_{s}^{X,N})\\
\leq &C\mathbb{E}\bigg(\frac{1}{N}\sum_{i=1}^N|X_{s}^i-X_{s}^{i,N}|^{2}\bigg)^{p/2}+C\mathbb{E}\mathbb{W}_{2}^p(\mathcal{L}_{X_{s}^i},\mu_{s}^{X})\\
\leq &C\mathbb{E}|X_{s}^i-X_{s}^{i,N}|^p+C\mathbb{E}\mathbb{W}_{2}^p(\mathcal{L}_{X_{s}^i},\mu_{s}^{X}).
\end{align*}
Putting $H_1(t)$-$H_3(t)$ together, we conclude that
\begin{align*}
 \mathbb{E}\bigg(\sup_{0\leq s\leq t}|\Upsilon_{s}^i|^p\bigg) \leq&C\mathbb{E}\int_{0}^t|X_{s}^i-X_{s}^{i,N}|^p{\rm d}s+C\mathbb{E}\int_{0}^t|X_{s-\tau}^i-X_{s-\tau}^{i,N}|^p{\rm d}s\\
 &+C\mathbb{E}\int_{0}^t\mathbb{W}_{2}^p(\mathcal{L}_{X_{s}^i},\mu_{s}^{X,N}){\rm d}s+C\mathbb{E}\int_{0}^t(t-s)^{2H-1}\mathbb{W}_{2}^p(\mathcal{L}_{X_{s}^i},\mu_{s}^{X,N}){\rm d}s\\
 \leq& \int_{0}^t\mathbb{E}\bigg(\sup_{0\leq s\leq r}|X_{s}^i-X_{s}^{i,N}|^p\bigg){\rm d}r+C\int_{0}^t(t-r)^{2H-1}\mathbb{E}\bigg(\sup_{0\leq s\leq r}|X_{s}^i-X_{s}^{i,N}|^p\bigg){\rm d}r\\
 &+C\int_{0}^t(t-r)^{2H-1}\mathbb{E}\mathbb{W}_{2}^p(\mathcal{L}_{X_{r}^i},\mu_{r}^{X}){\rm d}r.
\end{align*}
According to Lemma \ref{Maodebudengshi}, it derives
\begin{align*}
|X_{s}^i-X_{s}^{i,N}|^p
\leq \lambda|X_{s-\tau}^i-X_{s-\tau}^{i,N}|^p+\frac{1}{(1-\lambda)^{p-1}}|\Upsilon_{s}^i|^p.
\end{align*}
Noting that $0<\lambda<1$, then
\begin{align*}
\mathbb{E}\bigg(\sup_{0\leq s\leq t}|X_{s}^i-X_{s}^{i,N}|^p\bigg)\leq& C\mathbb{E}\bigg(\sup_{0\leq s\leq t} |\Upsilon_{s}^i|^p\bigg).
\end{align*}
That is,
\begin{align*}
&\mathbb{E}\bigg(\sup_{0\leq s\leq t}|X_{s}^i-X_{s}^{i,N}|^p\bigg)\\
\leq&\int_{0}^t\mathbb{E}\bigg(\sup_{0\leq s\leq r}|X_{s}^i-X_{s}^{i,N}|^p\bigg){\rm d}r+C\int_{0}^t(t-r)^{2H-1}\mathbb{E}\bigg(\sup_{0\leq s\leq r}|X_{s}^i-X_{s}^{i,N}|^p\bigg){\rm d}r\\
 &+C\int_{0}^t(t-r)^{2H-1}\mathbb{E}\mathbb{W}_{2}^p(\mathcal{L}_{X_{r}^i},\mu_{r}^{X}){\rm d}r.
\end{align*}
Since $\mathbb{W}_{2}^p(\mathcal{L}_{X_{s}^i},\mu_{s}^{X})$ is controlled by the Wasserstein distance estimate(\cite{FG15}, Theorem 1), for $2\le p<\bar{p}/2$, it follows
\begin{align*}
\mathbb{E}\mathbb{W}_{2}^p(\mathcal{L}_{X_{s}^i},\mu_{s}^{X}) \leq &C(\mathbb{E}|X_{s}^i|^{\bar{p}})^{p/{\bar{p}}}\times\left\{\begin{array}{ll}
      N^{-\frac{1}{2}},\ \ p>\frac{d}{2},\\
     N^{-\frac{1}{2}}\log(1+N),\ p=\frac{d}{2},\\
     N^{-\frac{p}{d}},\ p\in[2,\frac{d}{2}).
 \end{array}
 \right.
\end{align*}
Then, similar to \eqref{sss}, by Lemmas \ref{Hlemma} and \ref{propagationyoujie}, the desired result can be obtained.
\end{proof}

\section{The Tamed Theta EM Scheme}
In this section, we consider the tamed theta EM scheme to approximate solution of \eqref{interactingparticlesystem}. Since the drift coefficient exhibits super-linear growth, we define
\begin{align}\label{driftcoefficienttamed}
  b_{\Delta}(x,y,\mu)=\frac{b(x,y,\mu)}{1+\Delta^{\alpha}|b(x,y,\mu)|},
\end{align}
for $x,y\in\mathbb{R}^d$, $\mu\in\mathcal{P}_{2}(\mathbb{R}^d)$ and $\alpha\in(0,1/2]$. Assume that for any time $T>0$, there exist positive constants $M, m\in \mathbb{N}$ such that $\Delta=\frac{\tau}{m}=\frac{T}{M}$, where $\Delta\in (0,1)$ is the step-size. Consequently, for the time grid $t_{k}=k\Delta$, it holds that $t_{k+1}-\tau =t_{k+1-m}$. For $k=-m,\cdots, 0,$ we set $Y_{t_k}^{i,N}=\xi^{i}(t_{k})$. For $k=0,1,\cdots,M$, we now define the discrete tamed theta EM scheme for \eqref{interactingparticlesystem} as follows,
\begin{align}\label{Eulerscheme}
Y_{t_{k+1}}^{i,N}-D(Y_{t_{k+1-m}}^{i,N})=&Y_{t_{k}}^{i,N}-D(Y_{t_{k-m}}^{i,N})+\theta b_{\Delta}(Y_{t_{k+1}}^{i,N},Y_{t_{k+1-m}}^{i,N},\mu_{t_{k+1}}^{Y,N})\Delta\\
&+(1-\theta)b_{\Delta}(Y_{t_{k}}^{i,N},Y_{t_{k-m}}^{i,N},\mu_{t_{k}}^{Y,N})\Delta+\sigma(\mu_{t_{k}}^{Y,N})\Delta B_{t_{k}}^{H,i},\notag
\end{align}
where the empirical measures $\mu_{(\cdot)}^{Y,N}:=\frac{1}{N}\sum\limits_{i=1}^{N}\delta_{Y_{(\cdot)}^{i,N}}$ and $\Delta B_{t_{k}}^{H,i}=B_{t_{k+1}}^{H,i}-B_{t_{k}}^{H,i}$. Let
\begin{align}\label{barYdelta}
\bar{Y}_{\Delta}^{i,N}(t)=\left\{\begin{array}{ll}
\xi^{i}(t),\ \ -\tau\leq t\leq 0 , \\
\sum\limits_{k=0}^{[T/\Delta]}Y_{t_{k}}^{i,N} {\rm I}_{[k\Delta,(k+1)\Delta)}(t),\ \ t\geq 0,
\end{array}
\right.
\end{align}
and
\begin{align*}
\bar{Y}_{\Delta+}^{i,N}(t)=\left\{\begin{array}{ll}
\xi^{i}(t),\ \ -\tau\leq t\leq 0 , \\
\sum\limits_{k=0}^{[T/\Delta]}Y_{t_{k+1}}^{i,N} {\rm I}_{[k\Delta,(k+1)\Delta)}(t),\ \ t\geq 0,
\end{array}
\right.
\end{align*}
and the continuous tamed theta EM solution writes as
\begin{align}\label{continuoustimetamedthetascheme}
 Y_{\Delta}^{i,N}(t)=&D(\bar{Y}_{\Delta}^{i,N}(t))+\xi^{i}(0)-D(\xi^{i}(-\tau))+\theta\int_{0}^tb_{\Delta}(\bar{Y}_{\Delta+}^{i,N}(s),\bar{Y}_{\Delta+}^{i,N}(s-\tau),\bar{\mu}^{Y_{+},N}(s)){\rm d}s\\
 &+(1-\theta)\int_{0}^tb_{\Delta}(\bar{Y}_{\Delta}^{i,N}(s),\bar{Y}_{\Delta}^{i,N}(s-\tau),\bar{\mu}^{Y,N}(s)){\rm d}s+\int_{0}^t\sigma(\bar{\mu}^{Y,N}(s)){\rm d}B_{s}^{H,i}, \ \ t>0,\notag
\end{align}
where $\bar{\mu}^{Y,N}(\cdot):=\frac{1}{N}\sum\limits_{i=1}^N\delta_{\bar{Y}_{\Delta}^{i,N}}(\cdot)$, $\bar{\mu}^{Y_{+},N}(\cdot):=\frac{1}{N}\sum\limits_{i=1}^N\delta_{\bar{Y}_{\Delta+}^{i,N}}(\cdot)$ and $Y_{\Delta}^{i,N}(t)=\xi^{i}(t)$ for $t\in[-\tau,0]$. However, this $Y_{\Delta}^{i,N}(t)$ is not $\mathcal{F}_{t}$-adapted, it does not meet the fundamental requirement in the It\^{o} stochastic analysis. To avoid calculus, we use the discrete split-step theta scheme (see \cite{TY18}) described as follows: For $k=-m,\cdots,0$, set $Z_{t_{k}}^{i,N}=\xi^{i}(k\Delta)$. For $k=0,1,\cdots,M-1$, rewrite \eqref{Eulerscheme} as
\begin{align}\label{reformedtamedformula}
 \left\{\begin{array}{ll}
    Y_{t_{k}}^{i,N}=D(Y_{t_{k-m}}^{i,N})+Z_{t_{k}}^{i,N}-D(Z_{t_{k-m}}^{i,N})+\theta b_{\Delta}(Y_{t_{k}}^{i,N},Y_{t_{k-m}}^{i,N},\mu_{t_{k}}^{Y,N})\Delta, \\
  Z_{t_{k+1}}^{i,N}=D(Z_{t_{k+1-m}}^{i,N})+Z_{t_{k}}^{i,N}-D(Z_{t_{k-m}}^{i,N})+b_{\Delta}(Y_{t_{k}}^{i,N},Y_{t_{k-m}}^{i,N},\mu_{t_{k}}^{Y,N})\Delta+\sigma(\mu_{t_{k}}^{Y,N})\Delta B_{t_{k}}^{H,i},
 \end{array}
 \right.
\end{align}
This scheme can also be rewritten as
\begin{align*}
  &Z_{t_{k+1}}^{i,N}-D(Z_{t_{k+1-m}}^{i,N})\\
  =&Z_{t_{0}}^{i,N}-D(Z_{t_{-m}}^{i,N})+\sum_{j=0}^kb_{\Delta}(Y_{t_{j}}^{i,N},Y_{t_{j-m}}^{i,N},\mu_{t_{j}}^{Y,N})\Delta+\sum_{j=0}^k\sigma(\mu_{t_{j}}^{Y,N})\Delta B_{t_{j}}^{H,i} \\
 =&\xi^{i}(0)-D(\xi^i(-\tau))-\theta b_{\Delta}(\xi^i(0),\xi^i(-\tau),\delta_{\xi(0)})\Delta\\
 &+\sum_{j=0}^kb_{\Delta}(Y_{t_{j}}^{i,N},Y_{t_{j-m}}^{i,N},\mu_{t_{j}}^{Y,N})\Delta+\sum_{j=0}^k\sigma(\mu_{t_{j}}^{Y,N})\Delta B_{t_{j}}^{H,i}.
\end{align*}
Similarly, we define the corresponding continuous-time split-step tamed theta EM solution $Z_{\Delta}^{i,N}(t)$ as follows: For any $t\in[-\tau,0)$, $Z_{\Delta}^{i,N}(t)=\xi^i(t)$, $Z_{\Delta}^{i,N}(0)=\xi^i(0)-\theta b_{\Delta}(\xi^i(0)$, $\xi^i(-\tau),\delta_{\xi(0)})\Delta$. For any $t\in[0,T]$,
\begin{align}\label{Zdeltatformula}
  {\rm d}[Z_{\Delta}^{i,N}(t)-D(Z_{\Delta}^{i,N}(t-\tau))]=b_{\Delta}(\bar{Y}_{\Delta}^{i,N}(t),\bar{Y}_{\Delta}^{i,N}(t-\tau),\bar{\mu}^{Y,N}(t)){\rm d}t+\sigma(\bar{\mu}^{Y,N}(t)){\rm d}B_{t}^{H,i},
\end{align}
where $\bar{Y}_{\Delta}^{i,N}(t)$ is defined by \eqref{barYdelta}. Combining \eqref{continuoustimetamedthetascheme}, \eqref{reformedtamedformula} and \eqref{Zdeltatformula}, we see
\begin{align*}
  Y_{\Delta}^{i,N}(t)-D(Y_{\Delta}^{i,N}(t-\tau))-\theta b_{\Delta}(Y_{\Delta}^{i,N}(t),Y_{\Delta}^{i,N}(t-\tau),\mu^{Y,N}(t))\Delta=Z_{\Delta}^{i,N}(t)-D(Z_{\Delta}^{i,N}(t-\tau)).
\end{align*}
Let $\tilde{Y}_{\Delta}^{i,N}(t)=Y_{\Delta}^{i,N}(t)-D(Y_{\Delta}^{i,N}(t-\tau))-\theta b_{\Delta}(Y_{\Delta}^{i,N}(t),Y_{\Delta}^{i,N}(t-\tau),\mu^{Y,N}(t))\Delta$. We can rewrite \eqref{Zdeltatformula} as
\begin{align}\label{tildeYdelta}
\tilde{Y}_{\Delta}^{i,N}(t)=\tilde{Y}_{\Delta}^{i,N}(0)+\int_{0}^tb_{\Delta}(\bar{Y}_{\Delta}^{i,N}(s),\bar{Y}_{\Delta}^{i,N}(s-\tau),\bar{\mu}^{Y,N}(s)){\rm d}s+\int_{0}^t\sigma(\bar{\mu}^{Y,N}(s)){\rm d}B_{s}^{H,i}.
\end{align}
It is clear that $\tilde{Y}_{\Delta}^{i,N}$ coincide with $\bar{Y}_{\Delta}^{i,N}(t)-D(\bar{Y}_{\Delta}^{i,N}(t-\tau))-\theta b_{\Delta}(\bar{Y}_{\Delta}^{i,N}(t),\bar{Y}_{\Delta}^{i,N}(t-\tau),\bar{\mu}^{Y,N}(t))$ at grid points $t=k\Delta$, $k=0,1,\cdots,M-1$, this also means that the continuous-time tamed theta EM solution $Y_{\Delta}^{i,N}(t)$ and the discrete-time tamed theta EM solution $\bar{Y}_{\Delta}^{i,N}(t)$ take the same values at grid points $t=k\Delta$, $k=0,1,\cdots,M-1$.

\begin{lem}\label{bdeltayoujie}
{\rm There exists a positive constant $K_6\geq 1$ such that
\begin{align*}
  |b_{\Delta}(x,y,\mu)|\leq \min\{K_6\Delta^{-\alpha}(1+|x|+|y|+\mathbb{W}_{2}(\mu,\delta_{0})), |b(x,y,\mu)|\}
\end{align*}
for $x,y\in \mathbb{R}^d$ and $\mu\in\mathcal{P}_{2}(\mathbb{R}^d)$.}
\end{lem}
\begin{proof}
 For $\forall x,y>0$, by \eqref{driftcoefficienttamed}, it is easy to see that $|b_{\Delta}(x,y,\mu)|\leq |b(x,y,\mu)|$. On the other hand, we immediately get
 \begin{align*}
  |b_{\Delta}(x,y,\mu)|=\Delta^{-\alpha}\frac{|b(x,y,\mu)|}{\Delta^{-\alpha}+|b(x,y,\mu)|}\leq \Delta^{-\alpha}\leq K_6\Delta^{-\alpha}(1+|x|+|y|+\mathbb{W}_{2}(\mu,\delta_{0})).
 \end{align*}
\end{proof}

\begin{lem}\label{B2}
  {\rm Let Assumption \ref{danbiantiaojian} hold. One observes}
  \begin{align*}
   \langle x-D(y)-\bar{x}+D(\bar{y}),b_{\Delta}(x,y,\mu)-b_{\Delta}(\bar{x},\bar{y},\nu)\rangle\leq K_{7}(|x-\bar{x}|^2+|y-\bar{y}|^2+\mathbb{W}_{2}^2(\mu,\nu)),
  \end{align*}
  {\rm and }
  \begin{align*}
    |b(x,y,\mu)-b_{\Delta}(x,y,\mu)|^p\leq K_8\Delta^{\alpha p}(1+|x|^{2(l+1)p}+|y|^{2(l+1)p}+\mathbb{W}_{2}^{2p}(\mu,\delta_{0})),
  \end{align*}
  {\rm where $K_7, K_8$ are positive constants.}
\end{lem}
\begin{proof}
For the first inequality, we write
\begin{align*}
&b_{\Delta}(x,y,\mu)-b_{\Delta}(\bar{x},\bar{y},\nu)\\
=&\frac{b(x,y,\mu)-b(\bar{x},\bar{y},\nu)}{(1+\Delta^{\alpha}|b(x,y,\mu)|)(1+\Delta^{\alpha}|b(\bar{x},\bar{y},\nu)|)}\\
&+\frac{\Delta^{\alpha}[|b(\bar{x},\bar{y},\nu)|\cdot b(x,y,\mu)-|b(x,y,\mu)|\cdot b(\bar{x},\bar{y},\nu)]}{(1+\Delta^{\alpha}| b(x,y,\mu)|)(1+\Delta^{\alpha}|b(\bar{x},\bar{y},\nu)|)}.
\end{align*}
Using Assumption \ref{danbiantiaojian} for the first term, the Lipschitz continuity of $b$ and its growth condition from Remark \ref{remarkone} along with the Young's inequality for the second term, we obtain the desired result. Furthermore, for the second inequality, we observe that
\begin{align*}
|b(x,y,\mu)-b_{\Delta}(x,y,\mu)|^p &= \left|b(x,y,\mu)-\frac{b(x,y,\mu)}{1+\Delta^{\alpha}|b(x,y,\mu)|}\right|^p\\
&= \left|\frac{\Delta^{\alpha}|b(x,y,\mu)|\cdot b(x,y,\mu)}{1+\Delta^{\alpha}|b(x,y,\mu)|}\right|^p\\
&\leq \Delta^{\alpha p}|b(x,y,\mu)|^{2p}.
\end{align*}
Then, with the growth condition of $b$, the second part can be shown.
\end{proof}

\begin{rem}\label{B3}
 {\rm  Let Assumption \ref{danbiantiaojian} hold. We derive directly from Lemmas \ref{bdeltayoujie}, \ref{B2} that}
  \begin{align*}
    \langle x-D(y), b_{\Delta}(x,y,\mu)\rangle\leq C(1+|x|^2+|y|^2+\mathbb{W}_{2}^2(\mu,\delta_{0})).
  \end{align*}
\end{rem}

 \subsection{Moment Boundness}
 In order to get the desired results, we give the following Lemma.

\begin{lem}\label{tamedmomentbound}
{\rm Let Assumptions \ref{chushizhixidetiaojian}-\ref{sigmamanzudechushitiaojian} hold. For $\bar{p}\ge2$, we have}
\begin{align*}
\sup_{0\leq t\leq T}\mathbb{E}|Y_{\Delta}^{i,N}(t)|^{\bar{p}}\leq C,
\end{align*}
{\rm where the positive constant $C$ is independent of $\Delta$.}
\end{lem}

\begin{proof}
We begin by considering the case $\bar{p}\ge 4$, using moment estimation techniques for numerical solutions from \cite{S16}. The case $2 \le \bar{p} < 4$ then follows directly from the Young's inequality. For $a>0$, let $\lfloor a\rfloor$ be the integer part of $a$. Applying the fractional It\^{o} formula to $[1+|\tilde{Y}_{\Delta}^{i,N}(t)|^2]^{\bar{p}/2}$ and using the fact that $\mathbb{E}B_{s}^{H,i}=0$, it holds that
\begin{align*}
&\mathbb{E}[1+|\tilde{Y}_{\Delta}^{i,N}(t)|^2]^{\bar{p}/2}\\
\leq&\mathbb{E}[1+|\tilde{Y}_{\Delta}^{i,N}(0)|^2]^{\bar{p}/2}+\bar{p}\mathbb{E}\int_{0}^t[1+|\tilde{Y}_{\Delta}^{i,N}(s)|^2]^{\frac{\bar{p}-2}{2}}\langle \tilde{Y}_{\Delta}^{i,N}(s),b_{\Delta}(\bar{Y}_{\Delta}^{i,N}(s),\bar{Y}_{\Delta}^{i,N}(s),\bar{\mu}^{Y,N}(s))\rangle\\
&+\bar{p}(\bar{p}-2)H(2H-1)\mathbb{E}\int_{0}^t[1+|\tilde{Y}_{\Delta}^{i,N}(s)|^2]^{\frac{\bar{p}-2}{2}}\bigg(\int_{0}^s(s-r)^{2H-2}\|\sigma (\bar{\mu}^{Y,N}(r))\|^2{\rm d}r\bigg){\rm d}s\\
\leq &\mathbb{E}[1+|\tilde{Y}_{\Delta}^{i,N}(0)|^2]^{\bar{p}/2}+\bar{p}(\bar{p}-2)H(2H-1)\mathbb{E}\int_{0}^t[1+|\tilde{Y}_{\Delta}^{i,N}(s)|^2]^{\frac{\bar{p}-2}{2}}\\
&\times\bigg(\int_{0}^s(s-r)^{2H-2}\|\sigma (\bar{\mu}^{Y,N}(r))\|^2{\rm d}r\bigg){\rm d}s\\
&+\bar{p}\mathbb{E}\int_{0}^t[1+|\tilde{Y}_{\Delta}^{i,N}(s)|^2]^{\frac{\bar{p}-2}{2}}\langle \bar{Y}_{\Delta}^{i,N}(s)-D(\bar{Y}_{\Delta}^{i,N}(s-\tau)),b_{\Delta}(\bar{Y}_{\Delta}^{i,N}(s),\bar{Y}_{\Delta}^{i,N}(s),\bar{\mu}^{Y,N}(s))\rangle\\
&+\bar{p}\mathbb{E}\int_{0}^t[1+|\tilde{Y}_{\Delta}^{i,N}(s)|^2]^{\frac{\bar{p}-2}{2}}\langle \tilde{Y}_{\Delta}^{i,N}(s)-\hat{Y}_{\Delta}^{i,N}(s),b_{\Delta}(\bar{Y}_{\Delta}^{i,N}(s),\bar{Y}_{\Delta}^{i,N}(s),\bar{\mu}^{Y,N}(s))\rangle{\rm d}s\\
=:&\mathbb{E}[1+|\tilde{Y}_{\Delta}^{i,N}(0)|^2]^{\bar{p}/2}+G_{1}(t)+G_{2}(t)+G_{3}(t).
\end{align*}
where
\begin{align*}
\hat{Y}_{\Delta}^{i,N}(s)=\bar{Y}_{\Delta}^{i,N}(s)-D(\bar{Y}_{\Delta}^{i,N}(s-\tau))-\theta b_{\Delta}(\bar{Y}_{\Delta}^{i,N}(s),\bar{Y}_{\Delta}^{i,N}(s),\bar{\mu}^{Y,N}(s))\Delta.
\end{align*}
For $t\in[0,T]$, the Young's inequality, the Fubini theorem and Lemma \ref{bdeltayoujie} give that
\begin{align*}
  G_{1}(t)  \leq & C\mathbb{E}\int_{0}^t\bigg(1+|Y_{\Delta}^{i,N}(s)|^{\bar{p}}+|Y_{\Delta}^{i,N}(s-\tau)|^{\bar{p}}\\
  &+|\theta b_{\Delta}(Y_{\Delta}^{i,N}(s),Y_{\Delta}^{i,N}(s-\tau),\mu^{Y,N}(s))\Delta |^{\bar{p}}\bigg){\rm d}s+C\int_{0}^t(t-r)^{2H-1}\mathbb{E}|\bar{Y}_{\Delta}^{i,N}(r)|^{\bar{p}}{\rm d}r\\
  \leq &C\mathbb{E}\int_{0}^t\bigg(1+|Y_{\Delta}^{i,N}(s)|^{\bar{p}}+|Y_{\Delta}^{i,N}(s-\tau)|^{\bar{p}}+|\bar{Y}_{\Delta}^{i,N}(s)|^{\bar{p}}\bigg){\rm d}s+C\int_{0}^t(t-s)^{2H-1}\mathbb{E}|\bar{Y}_{\Delta}^{i,N}(s)|^{\bar{p}}{\rm d}s\\
  &+C\Delta^{(1-\alpha)\bar{p}}\mathbb{E}\int_{0}^t\bigg(1+|Y_{\Delta}^{i,N}(s)|^{\bar{p}}+|Y_{\Delta}^{i,N}(s-\tau)|^{\bar{p}}\bigg)\\
  \leq &C+C\int_{0}^t\sup_{0\leq u\leq s}\mathbb{E}|Y_{\Delta}^{i,N}(u)|^{\bar{p}}{\rm d}s+C\int_{0}^t(t-s)^{2H-1}\bigg(\sup_{0\leq u\leq s}\mathbb{E}|Y_{\Delta}^{i,N}(u)|^{\bar{p}}\bigg){\rm d}s.
\end{align*}
Remark \ref{B3} together with the Young's inequality yield
\begin{align*}
 G_{2}(t)\leq &C\mathbb{E}\int_{0}^t[1+|\tilde{Y}_{\Delta}^{i,N}(s)|^2]^{\frac{\bar{p}-2}{2}}\bigg(1+|\bar{Y}_{\Delta}^{i,N}(s)|^2+|\bar{Y}_{\Delta}^{i,N}(s-\tau)|^2+\mathbb{W}_{2}^2(\mu^{\bar{Y},N}(s),\delta_{0})\bigg){\rm d}s\\
 \leq &C\mathbb{E}\int_{0}^t\bigg([1+|\tilde{Y}_{\Delta}^{i,N}(s)|^2]^{\frac{\bar{p}}{2}}+1+|\bar{Y}_{\Delta}^{i,N}(s)|^{\bar{p}}+|\bar{Y}_{\Delta}^{i,N}(s-\tau)|^{\bar{p}}\bigg){\rm d}s\\
 \leq &C+C\int_{0}^t\sup_{0\leq u\leq s}\mathbb{E}|Y_{\Delta}^{i,N}(u)|^{\bar{p}}{\rm d}s.
\end{align*}
We may compute
\begin{align*}
G_{3}(t)  =&\bar{p}\mathbb{E}\int_{0}^t[1+|\hat{Y}_{\Delta}^{i,N}(s)|^2]^{\frac{\bar{p}-2}{2}}\langle \tilde{Y}_{\Delta}^{i,N}(s)-\hat{Y}_{\Delta}^{i,N}(s),b_{\Delta}(\bar{Y}_{\Delta}^{i,N}(s),\bar{Y}_{\Delta}^{i,N}(s),\bar{\mu}^{Y,N}(s))\rangle{\rm d}s\\
  &+p\mathbb{E}\int_{0}^t\bigg[[1+|\tilde{Y}_{\Delta}^{i,N}(s)|^2]^{\frac{\bar{p}-2}{2}}-[1+|\hat{Y}_{\Delta}^{i,N}(s)|^2]^{\frac{\bar{p}-2}{2}}\bigg]\\
  &\times\langle \tilde{Y}_{\Delta}^{i,N}(s)-\hat{Y}_{\Delta}^{i,N}(s),b_{\Delta}(\bar{Y}_{\Delta}^{i,N}(s),\bar{Y}_{\Delta}^{i,N}(s),\bar{\mu}^{Y,N}(s))\rangle{\rm d}s\\
  =:&\bar{p}G_{31}(t)+\bar{p}G_{32}(t),
\end{align*}
where
\begin{align*}
\tilde{Y}_{\Delta}^{i,N}(t)-\hat{Y}_{\Delta}^{i,N}(s)=\int_{\lfloor\frac{s}{\Delta}\rfloor\Delta}^sb_{\Delta}(\bar{Y}_{\Delta}^{i,N}(u),\bar{Y}_{\Delta}^{i,N}(u-\tau),\bar{\mu}^{Y,N}(u)){\rm d}u+\int_{\lfloor\frac{s}{\Delta}\rfloor\Delta}^s\sigma(\bar{\mu}^{Y,N}(u)){\rm d}B_{u}^{H,i}.
\end{align*}
We use Lemma \ref{bdeltayoujie} and the Young's inequality to get
\begin{align*}
&G_{31}(t)\\
=&\mathbb{E}\int_{0}^t[1+|\hat{Y}_{\Delta}^{i,N}(s)|^2]^{\frac{\bar{p}-2}{2}}\bigg\langle \int_{\lfloor\frac{s}{\Delta}\rfloor\Delta}^s b_{\Delta}(\bar{Y}_{\Delta}^{i,N}(u),\bar{Y}_{\Delta}^{i,N}(s-\tau),\bar{\mu}^{Y,N}(u)){\rm d}u,\\
&b_{\Delta}(\bar{Y}_{\Delta}^{i,N}(s),\bar{Y}_{\Delta}^{i,N}(s),\bar{\mu}^{Y,N}(s))\bigg\rangle{\rm d}s\\
&+\mathbb{E}\int_{0}^t[1+|\hat{Y}_{\Delta}^{i,N}(s)|^2]^{\frac{\bar{p}-2}{2}}\bigg\langle\int_{\lfloor\frac{s}{\Delta}\rfloor\Delta}^s\sigma(\bar{\mu}^{Y,N}(u)){\rm d}B_{u}^{H,i},b_{\Delta}(\bar{Y}_{\Delta}^{i,N}(s),\bar{Y}_{\Delta}^{i,N}(s),\bar{\mu}^{Y,N}(s))\bigg\rangle{\rm d}s\\
\leq &\mathbb{E}\int_{0}^t[1+|\hat{Y}_{\Delta}^{i,N}(s)|^2]^{\frac{\bar{p}-2}{2}} \int_{\lfloor\frac{s}{\Delta}\rfloor\Delta}^s |b_{\Delta}(\bar{Y}_{\Delta}^{i,N}(u),\bar{Y}_{\Delta}^{i,N}(s-\tau),\bar{\mu}^{Y,N}(u))|{\rm d}u\\
&\times |b_{\Delta}(\bar{Y}_{\Delta}^{i,N}(s),\bar{Y}_{\Delta}^{i,N}(s-\tau),\bar{\mu}^{Y,N}(s))|{\rm d}s\\
\leq &\Delta\mathbb{E}\int_{0}^t[1+|\hat{Y}_{\Delta}^{i,N}(s)|^2]^{\frac{\bar{p}-2}{2}}|b_{\Delta}(\bar{Y}_{\Delta}^{i,N}(s),\bar{Y}_{\Delta}^{i,N}(s),\bar{\mu}^{Y,N}(s))|^2{\rm d}s\\
\leq &C\Delta^{1-2\alpha}\mathbb{E}\int_{0}^t\bigg[(1+|\bar{Y}_{\Delta}^{i,N}(u)|^{\bar{p}}+|\bar{Y}_{\Delta}^{i,N}(s-\tau)|^{\bar{p}})+\Delta^{(1-\alpha)\bar{p}}(1+|\bar{Y}_{\Delta}^{i,N}(u)|^{\bar{p}}+|\bar{Y}_{\Delta}^{i,N}(s-\tau)|^{\bar{p}})\bigg]{\rm d}s\\
\leq &C+C\int_{0}^t\sup_{0\leq u\leq s}\mathbb{E}|Y_{\Delta}^{i,N}(u)|^{\bar{p}}{\rm d}s.
\end{align*}
Again, by the fractional It\^{o} formula
\begin{align*}
&[1+|\tilde{Y}_{\Delta}(s)|^2]^{\frac{\bar{p}-2}{2}}\\
\leq &[1+|\tilde{Y}_{\Delta}^{i,N}(0)|^2]^{\frac{\bar{p}-2}{2}}\\
&+(p-2)\int_{0}^s[1+|\tilde{Y}_{\Delta}^{i,N}(u)|^2]^{\frac{\bar{p}-4}{2}}\langle \tilde{Y}_{\Delta}^{i,N}(u),b_{\Delta}(\bar{Y}_{\Delta}^{i,N}(u),\bar{Y}_{\Delta}^{i,N}(u-\tau),\bar{\mu}^{Y,N}(u))\rangle{\rm d}u\\
&+\frac{1}{2}H(2H-1)(\bar{p}-2)(\bar{p}-3)\int_{0}^s[1+|\tilde{Y}_{\Delta}^{i,N}(r)|^2]^{\frac{\bar{p}-4}{2}}\int_{0}^r(r-u)^{2H-2}\|\sigma (\bar{\mu}^{Y,N}(u))\|^2{\rm d}u{\rm d}r\\
&+(\bar{p}-2)\int_{0}^s[1+|\tilde{Y}_{\Delta}^{i,N}(u)|^2]^{\frac{\bar{p}-4}{2}}\langle \tilde{Y}_{\Delta}^{i,N}(u),\sigma(\bar{\mu}^{Y,N}(u))\rangle{\rm d}B_{u}^{H,i}.
\end{align*}
Then, by the Young's inequality and the Fubini theorem, we have
\begin{align*}
&\mathbb{E}[1+|\hat{Y}_{\Delta}(s)|^2]^{\frac{\bar{p}-2}{2}}\\
\leq &[1+|\hat{Y}_{\Delta}^{i,N}(0)|^2]^{\frac{\bar{p}-2}{2}}\\
&+(\bar{p}-2)\mathbb{E}\int_{0}^{\lfloor\frac{s}{\Delta}\rfloor\Delta}[1+|\tilde{Y}_{\Delta}^{i,N}(u)|^2]^{\frac{\bar{p}-4}{2}}\langle \tilde{Y}_{\Delta}^{i,N}(u),b_{\Delta}(\bar{Y}_{\Delta}^{i,N}(u),\bar{Y}_{\Delta}^{i,N}(u-\tau),\bar{\mu}^{Y,N}(u))\rangle{\rm d}u\\
&+\frac{1}{2}H(2H-1)(\bar{p}-2)(\bar{p}-3)\bigg(\int_{0}^{\lfloor\frac{s}{\Delta}\rfloor\Delta}\frac{u^{2H-1}}{2H-1}\mathbb{E}[1+|\tilde{Y}_{\Delta}^{i,N}(u)|^2]^{\frac{\bar{p}}{2}}{\rm d}u\\
&+\int_{0}^{\lfloor\frac{s}{\Delta}\rfloor\Delta}\frac{(\lfloor\frac{s}{\Delta}\rfloor\Delta-u)^{2H-1}}{2H-1}\mathbb{E}|\bar{Y}_{\Delta}^{i,N}(u)|^{\bar p}{\rm d}u\bigg)\\
&+(\bar{p}-2)\mathbb{E}\int_{0}^{\lfloor\frac{s}{\Delta}\rfloor\Delta}[1+|\tilde{Y}_{\Delta}^{i,N}(u)|^2]^{\frac{\bar{p}-4}{2}}\langle \tilde{Y}_{\Delta}^{i,N}(u),\sigma(\bar{\mu}^{Y,N}(u))\rangle{\rm d}B_{u}^{H,i}.
\end{align*}
One immediately gets
\begin{align*}
G_{32}(t)\leq &(\bar{p}-2)\mathbb{E}\int_{0}^t\int_{\lfloor\frac{s}{\Delta}\rfloor\Delta}^{s}[1+|\tilde{Y}_{\Delta}^{i,N}(u)|^2]^{\frac{\bar{p}-4}{2}}\langle \tilde{Y}_{\Delta}^{i,N}(u),b_{\Delta}(\bar{Y}_{\Delta}^{i,N}(u),\bar{Y}_{\Delta}^{i,N}(u-\tau),\bar{\mu}^{Y,N}(u))\rangle{\rm d}u\\
&\times \langle \tilde{Y}_{\Delta}^{i,N}(s)-\hat{Y}_{\Delta}^{i,N}(s),b_{\Delta}(\bar{Y}_{\Delta}^{i,N}(s),\bar{Y}_{\Delta}^{i,N}(s-\tau),\bar{\mu}^{Y,N}(s))\rangle{\rm d}s\\
&+(\bar{p}-2)(\bar{p}-3)\mathbb{E}\int_{0}^t\int_{\lfloor\frac{s}{\Delta}\rfloor\Delta}^{s}\bigg(\frac{u^{2H-1}}{2H-1}\mathbb{E}[1+|\tilde{Y}_{\Delta}^{i,N}(u)|^2]^{\frac{\bar{p}}{2}}\\
&+\frac{(s-u)^{2H-1}-(\lfloor\frac{s}{\Delta}\rfloor\Delta-u)^{2H-1}}{2H-1}\mathbb{E}|\bar{Y}_{\Delta}^{i,N}(u)|^{\bar p}\bigg){\rm d}u\\
&\times \langle \tilde{Y}_{\Delta}^{i,N}(s)-\hat{Y}_{\Delta}^{i,N}(s),b_{\Delta}(\bar{Y}_{\Delta}^{i,N}(s),\bar{Y}_{\Delta}^{i,N}(s-\tau),\bar{\mu}^{Y,N}(s))\rangle{\rm d}s\\
&+(\bar{p}-2)\mathbb{E}\int_{0}^t\int_{\lfloor\frac{s}{\Delta}\rfloor\Delta}^{s}[1+|\tilde{Y}_{\Delta}^{i,N}(u)|^2]^{\frac{\bar{p}-4}{2}}\langle \tilde{Y}_{\Delta}^{i,N}(u),\sigma(\bar{\mu}^{Y,N}(u))\rangle{\rm d}B_{u}^{H,i}\\
&\times\langle \tilde{Y}_{\Delta}^{i,N}(s)-\hat{Y}_{\Delta}^{i,N}(s),b_{\Delta}(\bar{Y}_{\Delta}^{i,N}(s),\bar{Y}_{\Delta}^{i,N}(s-\tau),\bar{\mu}^{Y,N}(s))\rangle{\rm d}s\\
=:&(\bar{p}-2)G_{321}(t)+(\bar{p}-2)(\bar{p}-3)G_{322}(t)+(\bar{p}-2)G_{323}(t).
\end{align*}
By Lemma \ref{bdeltayoujie}, the Young's inequality, we get
\begin{align*}
G_{321}(t)\leq &\mathbb{E}\int_{0}^t\int_{\lfloor\frac{s}{\Delta}\rfloor\Delta}^{s}[1+|\tilde{Y}_{\Delta}^{i,N}(u)|^2]^{\frac{\bar{p}-4}{2}}\langle \tilde{Y}_{\Delta}^{i,N}(u),b_{\Delta}(\bar{Y}_{\Delta}^{i,N}(u),\bar{Y}_{\Delta}^{i,N}(u-\tau),\bar{\mu}^{Y,N}(u))\rangle{\rm d}u\\
&\times \bigg\langle \int_{\lfloor\frac{s}{\Delta}\rfloor\Delta}^{s}b_{\Delta}(\bar{Y}_{\Delta}^{i,N}(u),\bar{Y}_{\Delta}^{i,N}(u-\tau),\bar{\mu}^{Y,N}(u)){\rm d}u,b_{\Delta}(\bar{Y}_{\Delta}^{i,N}(s),\bar{Y}_{\Delta}^{i,N}(s-\tau),\bar{\mu}^{Y,N}(s))\bigg\rangle{\rm d}s\\
&+\mathbb{E}\int_{0}^t\int_{\lfloor\frac{s}{\Delta}\rfloor\Delta}^{s}[1+|\tilde{Y}_{\Delta}^{i,N}(u)|^2]^{\frac{\bar{p}-4}{2}}\langle \tilde{Y}_{\Delta}^{i,N}(u),b_{\Delta}(\bar{Y}_{\Delta}^{i,N}(u),\bar{Y}_{\Delta}^{i,N}(u-\tau),\bar{\mu}^{Y,N}(u))\rangle{\rm d}u\\
&\times\bigg\langle \int_{\lfloor\frac{s}{\Delta}\rfloor\Delta}^{s}\sigma(\bar{\mu}^{Y,N}(u)){\rm d}B_{u}^{H,i},b_{\Delta}(\bar{Y}_{\Delta}^{i,N}(s),\bar{Y}_{\Delta}^{i,N}(s-\tau),\bar{\mu}^{Y,N}(s))\bigg\rangle{\rm d}s\\
\leq &\Delta\mathbb{E}\int_{0}^t\int_{\lfloor\frac{s}{\Delta}\rfloor\Delta}^{s}[1+|\tilde{Y}_{\Delta}^{i,N}(u)|^2]^{\frac{\bar{p}-3}{2}}|b_{\Delta}(\bar{Y}_{\Delta}^{i,N}(s),\bar{Y}_{\Delta}^{i,N}(s-\tau),\bar{\mu}^{Y,N}(s))|^3{\rm d}u{\rm d}s\\
&+C\mathbb{E}\int_{0}^t\bigg[\bigg(\int_{\lfloor\frac{s}{\Delta}\rfloor\Delta}^{s}[1+|\tilde{Y}_{\Delta}^{i,N}(u)|^2]^{\frac{\bar{p}-3}{2}}|b_{\Delta}(\bar{Y}_{\Delta}^{i,N}(u),\bar{Y}_{\Delta}^{i,N}(u-\tau),\bar{\mu}^{Y,N}(u))|{\rm d}u\\
&\times|b_{\Delta}(\bar{Y}_{\Delta}^{i,N}(s),\bar{Y}_{\Delta}^{i,N}(s-\tau),\bar{\mu}^{Y,N}(s))|\bigg)^{\frac{\bar{p}}{\bar{p}-1}}+\bigg\lvert \int_{\lfloor\frac{s}{\Delta}\rfloor\Delta}^{s}\sigma(\bar{\mu}^{Y,N}(u)){\rm d}B_{u}^{H,i}\bigg\rvert^{\bar{p}}\bigg]{\rm d}s\\
\leq &C\Delta^{2-3\alpha}\mathbb{E}\int_{0}^t(1+|Y_{\Delta}^{i,N}(s)|^{\bar{p}}+|Y_{\Delta}^{i,N}(s-\tau)|^{\bar{p}}+|\bar{Y}_{\Delta}^{i,N}(s)|^p+|\bar{Y}_{\Delta}^{i,N}(s-\tau)|^{\bar{p}}){\rm d}s\\
&+C\Delta^{2-3\alpha}\Delta^{(1-\alpha)\bar{p}}\mathbb{E}\int_{0}^t(1+|Y_{\Delta}^{i,N}(s)|^{\bar{p}}+|Y_{\Delta}^{i,N}(s-\tau)|^p+|\bar{Y}_{\Delta}^{i,N}(s)|^{\bar{p}}\\
&+|\bar{Y}_{\Delta}^{i,N}(s-\tau)|^{\bar{p}}){\rm d}s\\
&+C\mathbb{E}\int_{0}^t\bigg[\bigg(\int_{\lfloor\frac{s}{\Delta}\rfloor\Delta}^{s}[1+|\tilde{Y}_{\Delta}^{i,N}(u)|^2]^{\frac{\bar{p}-3}{2}}|b_{\Delta}(\bar{Y}_{\Delta}^{i,N}(s),\bar{Y}_{\Delta}^{i,N}(s-\tau),\bar{\mu}^{Y,N}(s))|^2{\rm d}u|\bigg)^{\frac{\bar{p}}{\bar{p}-1}}\\
&+\mathbb{E}\int_{0}^t\int_{\lfloor\frac{s}{\Delta}\rfloor\Delta}^{s}\lvert \sigma(\bar{\mu}^{Y,N}(u))\rvert^{\bar{p}}{\rm d}u{\rm d}s\\
\leq &C+C\int_{0}^t\sup_{0\leq u\leq s}\mathbb{E}|Y_{\Delta}^{i,N}(u)|^{\bar{p}}{\rm d}s.
\end{align*}
Similarly, $G_{322}(t)$ can be estimated by Lemma \ref{Hlemma} as
\begin{align*}
G_{322}(t)\leq C+C\int_{0}^t\sup_{0\leq u\leq s}\mathbb{E}|Y_{\Delta}^{i,N}(u)|^{\bar{p}}{\rm d}s+C\int_{0}^t(t-s)^{2H-1}\bigg(\sup_{0\leq u\leq s}\mathbb{E}|Y_{\Delta}^{i,N}(u)|^{\bar{p}}\bigg){\rm d}s.
\end{align*}
Using Assumption \ref{danbiantiaojian}, Lemma \ref{bdeltayoujie} and \cite[Lemma 3.1.3]{BH008}, H\"{o}lder inequality, Young's inequality, we derive
\begin{align*}
&G_{323}(t)\\
=&\mathbb{E}\int_{0}^t\int_{\lfloor\frac{s}{\Delta}\rfloor\Delta}^{s}[1+|\tilde{Y}_{\Delta}^{i,N}(u)|^2]^{\frac{\bar{p}-4}{2}}\langle \tilde{Y}_{\Delta}^{i,N}(u),\sigma(\bar{\mu}^{Y,N}(u))\rangle{\rm d}B_{u}^{H,i}\\
&\times\bigg\langle \int_{\lfloor\frac{s}{\Delta}\rfloor\Delta}^{s}b_{\Delta}(\bar{Y}_{\Delta}^{i,N}(u),\bar{Y}_{\Delta}^{i,N}(u-\tau),\bar{\mu}^{Y,N}(u)){\rm d}u,b_{\Delta}(\bar{Y}_{\Delta}^{i,N}(s),\bar{Y}_{\Delta}^{i,N}(s-\tau),\bar{\mu}^{Y,N}(s))\bigg\rangle{\rm d}s\\
&+\mathbb{E}\int_{0}^t\int_{\lfloor\frac{s}{\Delta}\rfloor\Delta}^{s}[1+|\tilde{Y}_{\Delta}^{i,N}(u)|^2]^{\frac{\bar{p}-4}{2}}\langle \tilde{Y}_{\Delta}^{i,N}(u),\sigma(\bar{\mu}^{Y,N}(u))\rangle{\rm d}B_{u}^{H,i}\\
&\times\bigg\langle \int_{\lfloor\frac{s}{\Delta}\rfloor\Delta}^{s}\sigma(\bar{\mu}^{Y,N}(u)){\rm d}B_{u}^{H,i},b_{\Delta}(\bar{Y}_{\Delta}^{i,N}(s),\bar{Y}_{\Delta}^{i,N}(s-\tau),\bar{\mu}^{Y,N}(s))\bigg\rangle{\rm d}s\\
\leq &\int_{0}^t\bigg\{\bigg[\mathbb{E}\bigg(\int_{\lfloor\frac{s}{\Delta}\rfloor\Delta}^{s}[1+|\tilde{Y}_{\Delta}^{i,N}(u)|^2]^{\frac{\bar{p}-4}{2}}\langle \tilde{Y}_{\Delta}^{i,N}(u),\sigma(\bar{\mu}^{Y,N}(u))\rangle{\rm d}B_{u}^{H,i}\bigg)^2\bigg]^{1/2}\\
&\times\bigg[\mathbb{E}\bigg(\bigg\langle \int_{\lfloor\frac{s}{\Delta}\rfloor\Delta}^{s}b_{\Delta}(\bar{Y}_{\Delta}^{i,N}(u),\bar{Y}_{\Delta}^{i,N}(u-\tau),\bar{\mu}^{Y,N}(u)){\rm d}u,b_{\Delta}(\bar{Y}_{\Delta}^{i,N}(s),\bar{Y}_{\Delta}^{i,N}(s),\bar{\mu}^{Y,N}(s))\bigg\rangle\bigg)^2 \bigg]^{1/2}\bigg\}{\rm d}s\\
+ &\int_{0}^t\bigg\{\bigg[\mathbb{E}\bigg(\int_{\lfloor\frac{s}{\Delta}\rfloor\Delta}^{s}[1+|\tilde{Y}_{\Delta}^{i,N}(u)|^2]^{\frac{\bar{p}-4}{2}}\langle \tilde{Y}_{\Delta}^{i,N}(u),\sigma(\bar{\mu}^{Y,N}(u))\rangle{\rm d}B_{u}^{H,i}\bigg)^2\bigg]^{1/2}\\
&\times\bigg[\mathbb{E}\bigg(\bigg\langle \int_{\lfloor\frac{s}{\Delta}\rfloor\Delta}^{s}\sigma(\bar{\mu}^{Y,N}(u)){\rm d}B_{u}^{H,i},b_{\Delta}(\bar{Y}_{\Delta}^{i,N}(s),\bar{Y}_{\Delta}^{i,N}(s-\tau),\bar{\mu}^{Y,N}(s))\bigg\rangle\bigg)^2\bigg]^{1/2}\bigg\}{\rm d}s\\
\leq &2H{\Delta}^{2H-2}\mathbb{E}\int_{0}^t\int_{\lfloor\frac{s}{\Delta}\rfloor\Delta}^{s}[1+|\tilde{Y}_{\Delta}^{i,N}(u)|^2]^{\frac{\bar{p}-3}{2}}|\sigma(\bar{\mu}^{Y,N}(u))|^2{\rm d}u\\
&\times\bigg\{ \bigg(\mathbb{E}\Delta^{4-4\alpha}(1+|\bar{Y}_{\Delta}^{i,N}(s)|+|\bar{Y}_{\Delta}^{i,N}(s-\tau)|)^4\bigg)^{1/4}\bigg(\mathbb{E}\Delta^{-4\alpha}(1+|\bar{Y}_{\Delta}^{i,N}(s)|+|\bar{Y}_{\Delta}^{i,N}(s-\tau)|)^4\bigg)^{1/4}\\
&+\bigg(\mathbb{E}\Delta^4(1+|\bar{Y}_{\Delta}^{i,N}(s)|)^4\bigg)^{1/4}\bigg(\mathbb{E}\Delta^{-4\alpha}(1+|\bar{Y}_{\Delta}^{i,N}(s)|+|\bar{Y}_{\Delta}^{i,N}(s-\tau)|)^4\bigg)^{1/4}{\rm d}s\\
\leq & 2H{\Delta}^{2H-2}\mathbb{E}\int_{0}^t\bigg(\int_{\lfloor\frac{s}{\Delta}\rfloor\Delta}^{s}[1+|\tilde{Y}_{\Delta}^{i,N}(u)|^2]^{\frac{\bar{p}-3}{2}}|\sigma(\bar{\mu}^{Y,N}(u))|^2{\rm d}u\times C\Delta^{1-\alpha}\mathbb{E}\int_{0}^t(1+|Y_{\Delta}^{i,N}(s)|^{\bar{p}})\bigg){\rm d}s\\
\leq &C\Delta^{2H-1-\alpha}\int_{0}^{t}[C\Delta(1+\sup_{0\leq v\leq s}\mathbb{E}|Y_{\Delta}^{i,N}(v)|^{\bar{p}})]\times[C(1+\sup_{0\leq v\leq s}\mathbb{E}|Y_{\Delta}^{i,N}(v)|^{\bar{p}})]\mathrm{d}s\\
=&C\Delta^{2H-\alpha}\int_{0}^{t}(1+\sup_{0\leq v\leq s}\mathbb{E}|Y_{\Delta}^{i,N}(v)|^{\bar{p}})^{2}\mathrm{d}s\\
\leq & C\Delta^{2H-1-\alpha}(1+\sup_{0\leq v\leq t}\mathbb{E}|Y_{\Delta}^{i,N}(v)|^{\bar{p}})\int_{0}^{t}(1+\sup_{0\leq v\leq s}\mathbb{E}|Y_{\Delta}^{i,N}(v)|^{\bar{p}})\mathrm{d}s\\
\leq &C+C\int_{0}^t\sup_{0\leq u\leq s}\mathbb{E}|Y_{\Delta}^{i,N}(u)|^{\bar{p}}{\rm d}s.
\end{align*}
To sum up, we conclude
\begin{align*}
G_{3}(t)\leq C+C\int_{0}^t\sup_{0\leq u\leq s}\mathbb{E}|Y_{\Delta}^{i,N}(u)|^{\bar{p}}{\rm d}s+C\int_{0}^t(t-s)^{2H-1}\bigg(\sup_{0\leq u\leq s}\mathbb{E}|Y_{\Delta}^{i,N}(u)|^{\bar{p}}\bigg){\rm d}s.
\end{align*}
Combining $G_{1}(t)-G_{3}(t)$ together, it shows
\begin{align}\label{20240730}
\sup_{0\leq u\leq t}\mathbb{E}|\tilde{Y}_{\Delta}^{i,N}(u)|^{\bar{p}}\leq& \sup_{0\leq u\leq t}\mathbb{E}[1+|\tilde{Y}_{\Delta}^{i,N}(u)|^2]^{\bar{p}/2}\leq C+C\int_{0}^t\sup_{0\leq u\leq s}\mathbb{E}|Y_{\Delta}^{i,N}(u)|^{\bar{p}}{\rm d}s\\
&+C\int_{0}^t(t-s)^{2H-1}\bigg(\sup_{0\leq u\leq s}\mathbb{E}|Y_{\Delta}^{i,N}(u)|^{\bar{p}}\bigg){\rm d}s.\notag
\end{align}
Noting that $|x-y|^{\bar{p}}\geq 2^{1-\bar{p}}|x|^{\bar{p}}-|y|^{\bar{p}}$, Assumption \ref{zhonglixiangdetiaojian} and Lemma \ref{bdeltayoujie} give
\begin{align*}
|\tilde{Y}_{\Delta}^{i,N}(u)|^{\bar{p}}\geq &2^{1-\bar{p}}|Y_{\Delta}^{i,N}(t)-D(Y_{\Delta}^{i,N}(t-\tau))|^{\bar{p}}-|\theta b_{\Delta}(Y_{\Delta}^{i,N}(t),Y_{\Delta}^{i,N}(t-\tau),\mu^{Y,N}(t))\Delta|^{\bar{p}}\\
\geq &2^{1-\bar{p}}[2^{1-\bar{p}}|Y_{\Delta}^{i,N}(t)|^{\bar{p}}-|D(Y_{\Delta}^{i,N}(t-\tau))|^{\bar{p}}]-|\theta b_{\Delta}(Y_{\Delta}^{i,N}(t),Y_{\Delta}^{i,N}(t-\tau),\mu^{Y,N}(t))\Delta|^{\bar{p}}\\
\geq &(2^{2-2\bar{p}}-4^{\bar{p}-1}\theta^{\bar{p}}K_6^{\bar{p}}\Delta^{(1-\alpha) \bar{p}})|Y_{\Delta}^{i,N}(t)|^{\bar{p}}-(2^{1-\bar{p}}\lambda^{\bar{p}}+4^{\bar{p}-1}\theta^pK_6^{\bar{p}}\Delta^{(1-\alpha) \bar{p}})|Y_{\Delta}^{i,N}(t-\tau)|^{\bar{p}}\\
&-4^{\bar{p}-1}\theta^{\bar{p}}K_6^{\bar{p}}\Delta^{(1-\alpha) \bar{p}}\mathbb{W}_{2}^{\bar{p}}(\mu^{Y,N}(t),\delta_{0})-4^{\bar{p}-1}\theta^{\bar{p}}K_6^{\bar{p}}\Delta^{(1-\alpha) \bar{p}}.
\end{align*}
Thus, we obtain
\begin{align*}
&\sup_{0\leq u\leq t}\mathbb{E}|Y_{\Delta}^{i,N}(t)|^{\bar{p}}\leq (2^{2-2\bar{p}}-3\times 4^{\bar{p}-1}\theta^{\bar{p}}K_6^{\bar{p}}\Delta^{(1-\alpha) \bar{p}}-2^{1-\bar{p}}\lambda^{\bar{p}})^{-1}\\
&\times\bigg[\sup_{0\leq u\leq t}\mathbb{E}|\tilde{Y}_{\Delta}^{i,N}(u)|^{\bar{p}}+4^{\bar{p}-1}\theta^{\bar{p}}K_6^{\bar{p}}\Delta^{(1-\alpha) \bar{p}}+(2^{1-\bar{p}}\lambda^{\bar{p}}+4^{\bar{p}-1}\theta^{\bar{p}}K_6^{\bar{p}}\Delta^{(1-\alpha) \bar{p}})\mathbb{E}\|\xi\|^{\bar{p}}\bigg].
\end{align*}
Further, it follows from \eqref{20240730} that
\begin{align*}
\sup_{0\leq u\leq t}\mathbb{E}|Y_{\Delta}^{i,N}(u)|^{\bar{p}}\leq C+C\int_{0}^t\sup_{0\leq u\leq s}\mathbb{E}|Y_{\Delta}^{i,N}(u)|^{\bar{p}}{\rm d}s+C\int_{0}^t(t-s)^{2H-1}\bigg(\sup_{0\leq u\leq s}\mathbb{E}|Y_{\Delta}^{i,N}(u)|^{\bar{p}}\bigg){\rm d}s.
\end{align*}
Consequently, the desired result can be shown by applying Lemma \ref{Hlemma}. This completes the proof.
\end{proof}

 \subsection{Strong Convergence Rate}
\begin{lem}\label{discreteandcontinuoustamedconvergence}
{\rm Let Assumptions \ref{chushizhixidetiaojian}-\ref{sigmamanzudechushitiaojian} hold. For $p\ge2$, then}
\begin{align*}
 \mathbb{E}\bigg(\sup_{0\leq k\leq M-1}\sup_{t_{k}\leq t\leq t_{k+1}}|Y_{\Delta}^{i,N}(t)-Y_{\Delta}^{i,N}(t_k)|^p\bigg)\leq C\Delta^{[(1-\alpha)\wedge H]p},
\end{align*}
{\rm where $C$ is a positive constant independent of $\Delta$.}
\end{lem}

\begin{proof}
For $t\in[t_k,t_{k+1}]$, by the elementary inequality, we may compute
\begin{align*}
&\mathbb{E}\bigg(\sup_{t_{k}\leq t\leq t_{k+1}} |\tilde{Y}_{\Delta}^{i,N}(t)-\tilde{Y}_{\Delta}^{i,N}(t_k)|^p\bigg)\\
=&\mathbb{E}\bigg(\sup_{t_{k}\leq t\leq t_{k+1}}\bigg\lvert \int_{t_k}^tb_{\Delta}(\bar{Y}_{\Delta}^{i,N}(s),\bar{Y}_{\Delta}^{i,N}(s-\tau),\bar{\mu}^{Y,N}(s)){\rm d}s+\int_{t_k}^t\sigma(\bar{\mu}^{Y,N}(s)){\rm d}B_{s}^{H,i}\bigg\rvert^p\bigg)\\
\leq &2^{p-1}\mathbb{E}\bigg(\sup_{t_{k}\leq t\leq t_{k+1}}\bigg\lvert \int_{t_k}^tb_{\Delta}(\bar{Y}_{\Delta}^{i,N}(s),\bar{Y}_{\Delta}^{i,N}(s-\tau),\bar{\mu}^{Y,N}(s)){\rm d}s\bigg\rvert^p\bigg)\\
&+2^{p-1}\mathbb{E}\bigg(\sup_{t_{k}\leq t\leq t_{k+1}}\bigg\lvert \int_{t_k}^t\sigma(\bar{\mu}^{Y,N}(s)){\rm d}B_{s}^{H,i}\bigg\rvert^p\bigg)\\
:=&W_{1}(t)+W_{2}(t).
\end{align*}
With the H\"{o}lder inequality and Lemmas \ref{bdeltayoujie}, \ref{tamedmomentbound}, it follows
\begin{align*}
W_{1}(t)=&2^{p-1}\mathbb{E}\bigg(\sup_{t_{k}\leq t\leq t_{k+1}}\bigg\lvert \int_{t_k}^tb_{\Delta}(\bar{Y}_{\Delta}^{i,N}(s),\bar{Y}_{\Delta}^{i,N}(s-\tau),\bar{\mu}^{Y,N}(s)){\rm d}s\bigg\rvert^p\bigg)\\
\leq &2^{p-1}\Delta^{p-1}\mathbb{E}\int_{t_k}^{t_{k+1}}|b_{\Delta}(\bar{Y}_{\Delta}^{i,N}(s),\bar{Y}_{\Delta}^{i,N}(s-\tau),\bar{\mu}^{Y,N}(s))|^p{\rm d}s\\
\leq &2^{p-1}\Delta^{p-1}\mathbb{E}\int_{t_k}^{t_{k+1}}K_{6}\Delta^{-\alpha p}(1+|\bar{Y}_{\Delta}^{i,N}(s)|+|\bar{Y}_{\Delta}^{i,N}(s-\tau)|+\mathbb{W}_{2}^p(\bar{\mu}^{Y,N}(s),\delta_{0})){\rm d}s\\
\leq &C\Delta^{(1-\alpha)p}+C\Delta^{(1-\alpha)p}\mathbb{E}\int_{t_k}^{t_{k+1}}(|\bar{Y}_{\Delta}^{i,N}(s)|^p+|\bar{Y}_{\Delta}^{i,N}(s-\tau)|^p){\rm d}s\\
\leq &C\Delta^{(1-\alpha)p}.
\end{align*}
Using the same techniques as the proof process of \cite[Theorem 3.1]{FHS22}, along with the H\"{o}lder inequality, the Fubini theorem and Lemmas \ref{sigmamanzudechushitiaojian}, \ref{tamedmomentbound}, we derive
\begin{align*}
 W_{2}(t)=&2^{p-1}\mathbb{E}\bigg(\sup_{t_{k}\leq t\leq t_{k+1}}\bigg\lvert \int_{t_k}^t\sigma(\bar{\mu}^{Y,N}(s)){\rm d}B_{s}^{H,i}\bigg\rvert^p\bigg)\\
\leq &C \Delta^{pH-1}\mathbb{E}\int_{t_{k}}^{t_{k+1}}(1+\mathbb{W}_{2}(\bar{\mu}_{s}^{Y,N},\delta_{0}))^p{\rm d}s
\leq C\Delta^{pH}.
\end{align*}
Recalling that $\alpha\in(0,1/2]$, then we conclude
\begin{align*}
\mathbb{E}\bigg(\sup_{t_{k}\leq t\leq t_{k+1}} |\tilde{Y}_{\Delta}^{i,N}(t)-\tilde{Y}_{\Delta}^{i,N}(t_k)|^p\bigg)\leq C\Delta^{[(1-\alpha )\wedge H]p}.
\end{align*}
Denote $\tilde{D}(t,t_k):=D(Y_{\Delta}^{i,N}(t-\tau))-D(Y_{\Delta}^{i,N}(t_k-\tau))$, and $\tilde{b}(t,t_k):=b_{\Delta}(Y_{\Delta}^{i,N}(t),Y_{\Delta}^{i,N}(t-\tau),\mu^{Y,N}(t))-b_{\Delta}(Y_{\Delta}^{i,N}(t_k),Y_{\Delta}^{i,N}(t_k-\tau),\mu^{Y,N}(t_{k}))$, together with Assumption \ref{zhonglixiangdetiaojian}
\begin{align}\label{tildeYbudengshi}
 &|\tilde{Y}_{\Delta}^{i,N}(t)-\tilde{Y}_{\Delta}^{i,N}(t_k)|^p\\
 \geq & 2^{1-p}|Y_{\Delta}^{i,N}(t)-Y_{\Delta}^{i,N}(t_k)-\tilde{D}(t,t_k)|^p-\theta^p\Delta^p|\tilde{b}(t,t_k)|^p\notag\\
 \geq &2^{2-2p}|Y_{\Delta}^{i,N}(t)-Y_{\Delta}^{i,N}(t_k)|^p-2^{1-p}\lambda^p|Y_{\Delta}^{i,N}(t-\tau)-Y_{\Delta}^{i,N}(t_k-\tau)|^p-\theta^p\Delta^p|\tilde{b}(t,t_k)|^p.\notag
\end{align}
For $0\leq t<t_1=\Delta$, then $t-\tau<t_1-\tau<0$, we derive from \eqref{tildeYbudengshi} and Lemma \ref{tamedmomentbound},
\begin{align*}
&\mathbb{E}\bigg(\sup_{0\leq t\leq t_1}|Y_{\Delta}^{i,N}(t)-Y_{\Delta}^{i,N}(t_0)|^p\bigg)
\leq C\mathbb{E}\bigg(\sup_{0\leq t\leq t_1}|\tilde{Y}_{\Delta}^{i,N}(t)-\tilde{Y}_{\Delta}^{i,N}(t_0)|^p\bigg)+C\Delta^{(1-\alpha)p}\leq C\Delta^{[(1-\alpha)\wedge H]p}.
\end{align*}
Similarly, for $t_1\leq t<t_2$, we obtain
\begin{align*}
&\mathbb{E}\bigg(\sup_{t_1\leq t\leq t_2}|Y_{\Delta}^{i,N}(t)-Y_{\Delta}^{i,N}(t_1)|^p\bigg)\\
\leq & C\mathbb{E}\bigg(\sup_{t_1\leq t\leq t_2}|\tilde{Y}_{\Delta}^{i,N}(t)-\tilde{Y}_{\Delta}^{i,N}(t_1)|^p\bigg)+C\mathbb{E}\bigg(\sup_{0\leq t\leq (t_{2-m}\vee0)}|\tilde{Y}_{\Delta}^{i,N}(t)-\tilde{Y}_{\Delta}^{i,N}(t_{1-m})|^p\bigg)+C\Delta^{(1-\alpha)p}\\
\leq &C\Delta^{[(1-\alpha)\wedge H]p}.
\end{align*}
By induction, the desired assertion can be shown.
\end{proof}

\begin{lem}\label{strongconvergenceofxiN-YiN}
{\rm Let Assumptions \ref{chushizhixidetiaojian}-\ref{sigmamanzudechushitiaojian} hold.  Fix any $\bar{p}>4(l+1)$ in Lemma \ref{tamedmomentbound}. Then it holds that for $2\le p<\bar{p}/2(l+1)$,
\begin{align*}
\mathbb{E}\bigg(\sup_{0\leq t\leq T}|X_t^{i,N}-Y_{\Delta}^{i,N}(t)|^p\bigg)\leq C\Delta^{\alpha p},
\end{align*}
where $C$ is a positive constant independent of $\Delta$.  }
\end{lem}

 \begin{proof}
For $t\in[0,T]$, denote $\Lambda(t)=\tilde{Y}_{\Delta}^{i,N}(t)-X_t^{i,N}+D(X_{t-\tau}^{i,N})$, then
 \begin{align*}
\Lambda(t)=&\Lambda(0)+\int_{0}^t(b_{\Delta}(\bar{Y}_{\Delta}^{i,N}(s),\bar{Y}_{\Delta}^{i,N}(s-\tau),\bar{\mu}^{Y,N}(s))-b(X_{s}^{i,N},X_{s-\tau}^{i,N},\mu_{s}^{X,N})){\rm d}s\\
&+\int_{0}^t(\sigma(\bar{\mu}^{Y,N}(s))-\sigma(\mu_{s}^{X,N})){\rm d}B_{s}^{H,i},
 \end{align*}
 where $\Lambda({0})=-\theta b_{\Delta}(\xi(0),\xi(-\tau),\delta_{0})\Delta$. The fractional It\^{o} formula yields,
 \begin{align*}
 |\Lambda(t)|^p
 \leq &|\Lambda(0)|^p\\
 &+p\int_{0}^t|\Lambda(s)|^{p-2}\langle \Lambda(s),b_{\Delta}(\bar{Y}_{\Delta}^{i,N}(s),\bar{Y}_{\Delta}^{i,N}(s-\tau),\bar{\mu}^{Y,N}(s))-b(X_{s}^{i,N},X_{s-\tau}^{i,N},\mu_{s}^{X,N})\rangle{\rm d}s\\
 &+H(2H-1)p(p-1)\int_{0}^t|\Lambda(s)|^{p-2}\int_{0}^s(s-r)^{2H-2}\|\sigma(\bar{\mu}^{Y,N}(r))-\sigma(\mu_{r}^{X,N})\|^2{\rm d}r{\rm d}s\\
 &+p\int_{0}^t|\Lambda(s)|^{p-2}\langle \Lambda(s),\sigma(\bar{\mu}^{Y,N}(s))-\sigma(\mu_{s}^{X,N})\rangle{\rm d}B_{s}^{H,i}\\
 \leq &|\Lambda(0)|^p+\sum_{i=1}^6Z_{i}(t),
 \end{align*}
 where
 \begin{align*}
 Z_{1}(t)=&p\int_{0}^t|\Lambda(s)|^{p-2}\langle \Lambda(s),b_{\Delta}(\bar{Y}_{\Delta}^{i,N}(s),\bar{Y}_{\Delta}^{i,N}(s-\tau),\bar{\mu}^{Y,N}(s))\\
 &-b(\bar{Y}_{\Delta}^{i,N}(s),\bar{Y}_{\Delta}^{i,N}(s-\tau),\bar{\mu}^{Y,N}(s))\rangle{\rm d}s\\
 Z_{2}(t)=&p\int_{0}^t|\Lambda(s)|^{p-2}\langle \Lambda(s),b(\bar{Y}_{\Delta}^{i,N}(s),\bar{Y}_{\Delta}^{i,N}(s-\tau),\bar{\mu}^{Y,N}(s))\\
 &-b(Y_{\Delta}^{i,N}(s),Y_{\Delta}^{i,N}(s-\tau),\mu^{Y,N}(s))\rangle{\rm d}s\\
 Z_{3}(t)=&p\int_{0}^t|\Lambda(s)|^{p-2}\langle \Lambda(s),b(Y_{\Delta}^{i,N}(s),Y_{\Delta}^{i,N}(s-\tau),\mu^{Y,N}(s))-b(X_{s}^{i,N},X_{s-\tau}^{i,N},\mu_{s}^{X,N})\rangle{\rm d}s,\\
 Z_{4}(t)=&H(2H-1)p(p-1)\int_{0}^t|\Lambda(s)|^{p-2}\int_{0}^s(s-r)^{2H-2}\|\sigma(\bar{\mu}^{Y,N}(r))-\sigma(\mu^{Y,N}(r))\|^2{\rm d}r{\rm d}s\\
 Z_{5}(t)=&H(2H-1)p(p-1)\int_{0}^t|\Lambda(s)|^{p-2}\int_{0}^s(s-r)^{2H-2}\|\sigma(\mu^{Y,N}(r))-\sigma(\mu_{r}^{X,N})\|^2{\rm d}r{\rm d}s\\
 Z_{6}(t)=&p\int_{0}^t|\Lambda(s)|^{p-2}\langle \Lambda(s),\sigma(\bar{\mu}^{Y,N}(s))-\sigma(\mu_{s}^{X,N})\rangle{\rm d}B_{s}^{H,i}.
 \end{align*}
 By Assumption \ref{zhonglixiangdetiaojian}, Lemmas \ref{bdeltayoujie}, \ref{B2}, \ref{tamedmomentbound}, and the Young's inequality, the H\"{o}lder inequality, we get
 \begin{align*}
   &\mathbb{E}\bigg(\sup_{0\leq u\leq t}Z_{1}(u)\bigg)\\
   \leq& C\mathbb{E}\int_{0}^t|\Lambda(s)|^p{\rm d}s+C\mathbb{E}\int_{0}^t|b_{\Delta}(\bar{Y}_{\Delta}^{i,N}(s),\bar{Y}_{\Delta}^{i,N}(s-\tau),\bar{\mu}^{Y,N}(s))\\
   &-b(\bar{Y}_{\Delta}^{i,N}(s),\bar{Y}_{\Delta}^{i,N}(s-\tau),\bar{\mu}^{Y,N}(s))|^p{\rm d}s\\
 \leq &C\mathbb{E}\int_{0}^t\bigg(|Y_{\Delta}^{i,N}(s)-X_{s}^{i,N}|^p+|Y_{\Delta}^{i,N}(s-\tau)-X_{s-\tau}^{i,N}|^p\\
 &+\theta^p\Delta^p|b_{\Delta}(Y_{\Delta}^{i,N}(s),Y_{\Delta}^{i,N}(s-\tau),\mu^{Y,N}(s))|^p\bigg){\rm d}s\\
 &+C\Delta^{\alpha p}\mathbb{E}\int_{0}^t(1+|\bar{Y}_{\Delta}^{i,N}(s)|^{2(l+1)p}+|\bar{Y}_{\Delta}^{i,N}(s-\tau)|^{2(l+1)p}+\mathbb{W}_{2}^{2p}(\bar{\mu}^{Y,N},\delta_0)){\rm d}s\\
 \leq &C\int_{0}^t\mathbb{E}\left(\sup_{0\leq u\leq s}|Y_{\Delta}^{i,N}(u)-X_{u}^{i,N}|^p\right){\rm d}s+C\Delta^{(1-\alpha)p}+C\Delta^{\alpha p},
 \end{align*}
where for the last term we used the fact that $2(l+1)p<\bar{p}$. Combining the elementary inequality, the Young's inequality, the H\"{o}lder inequality, Assumptions \ref{zhonglixiangdetiaojian}, \ref{danbiantiaojian} and Lemmas \ref{bdeltayoujie}, \ref{tamedmomentbound}, \ref{discreteandcontinuoustamedconvergence}, it follows
 \begin{align*}
 &\mathbb{E}\bigg(\sup_{0\leq u\leq t}Z_{2}(u)\bigg)\\
   \leq &C\mathbb{E}\int_{0}^t|\Lambda(s)|^p{\rm d}s+C\mathbb{E}\int_{0}^t
   |b(\bar{Y}_{\Delta}^{i,N}(s),\bar{Y}_{\Delta}^{i,N}(s-\tau),\bar{\mu}^{Y,N}(s))-b(Y_{\Delta}^{i,N}(s),Y_{\Delta}^{i,N}(s-\tau),\mu^{Y,N}(s))|^p{\rm d}s\\
   \leq &C\mathbb{E}\int_{0}^t\bigg(|Y_{\Delta}^{i,N}(s)-X_{s}^{i,N}|^p+|Y_{\Delta}^{i,N}(s-\tau)-X_{s-\tau}^{i,N}|^p+\theta^p\Delta^p|b_{\Delta}(Y_{\Delta}^{i,N}(s),Y_{\Delta}^{i,N}(s-\tau),\mu^{Y,N}(s))|\bigg){\rm d}s\\
 &+C\mathbb{E}\int_{0}^t\bigg[(1+|\bar{Y}_{\Delta}^{i,N}(s)|^l+|\bar{Y}_{\Delta}^{i,N}(s-\tau)|^l+|Y_{\Delta}^{i,N}(s)|^l+|Y_{\Delta}^{i,N}(s-\tau)|^l)^p\\
 &\times(|\bar{Y}_{\Delta}^{i,N}(s)-Y_{\Delta}^{i,N}(s)|+|\bar{Y}_{\Delta}^{i,N}(s-\tau)-Y_{\Delta}^{i,N}(s-\tau)|+\mathbb{W}_2^p(\bar{\mu}^{Y,N}(s),\mu^{Y,N})(s))\bigg]{\rm d}s\\
 \leq&C\mathbb{E}\int_{0}^t\bigg(|Y_{\Delta}^{i,N}(s)-X_{s}^{i,N}|^p+|Y_{\Delta}^{i,N}(s-\tau)-X_{s-\tau}^{i,N}|^p\bigg){\rm d}s+C\Delta^{(1-\alpha)p}\\
 &+C\int_{0}^t[\mathbb{E}(1+|\bar{Y}_{\Delta}^{i,N}(s)|^l+|\bar{Y}_{\Delta}^{i,N}(s-\tau)|^l+|Y_{\Delta}^{i,N}(s)|^l+|Y_{\Delta}^{i,N}(s-\tau)|^l)^{2p}]^{1/2}\\
 &\times[\mathbb{E}(|\bar{Y}_{\Delta}^{i,N}(s)-Y_{\Delta}^{i,N}(s)|+|\bar{Y}_{\Delta}^{i,N}(s-\tau)-Y_{\Delta}^{i,N}(s-\tau)|)^{2p}]^{1/2}{\rm d}s\\
 &+C\int_{0}^t\mathbb{E}(|\bar{Y}_{\Delta}^{i,N}(s)-Y_{\Delta}^{i,N}(s)|)^p{\rm d}s\\
 \leq &C\int_{0}^t\mathbb{E}\left(\sup_{0\leq u\leq s}|Y_{\Delta}^{i,N}(u)-X_{u}^{i,N}|^p\right){\rm d}s+C\Delta^{(1-\alpha)p}+C\Delta^{[(1-\alpha)\wedge H]p}.
 \end{align*}
 Again, by Assumption \ref{danbiantiaojian} and Lemmas \ref{bdeltayoujie}, \ref{tamedmomentbound}
 \begin{align*}
 &\mathbb{E}\bigg(\sup_{0\leq u\leq t}Z_{3}(u)\bigg)\\
 \leq &C\mathbb{E}\int_{0}^t|\Lambda(s)|^{p-2}[|Y_{\Delta}^{i,N}(s)-X_{s}^{i,N}|^2+|Y_{\Delta}^{i,N}(s-\tau)-X_{s-\tau}^{i,N}|^2+\mathbb{W}_2^2(\mu^{Y,N}(s),\mu^{X,N}(s))]{\rm d}s\\
 &+C\mathbb{E}\int_{0}^t|\Lambda(s)|^{p-2}|\theta b_{\Delta}(Y_{\Delta}^{i,N}(s),Y_{\Delta}^{i,N}(s-\tau),\mu^{Y,N}(s))\Delta|\\
 &\cdot|b(Y_{\Delta}^{i,N}(s),Y_{\Delta}^{i,N}(s-\tau),\mu^{Y,N}(s))-b(X_{s}^{i,N},X_{s-\tau}^{i,N},\mu_{s}^{X,N})|{\rm d}s\\
 \leq &C\mathbb{E}\int_{0}^t[|Y_{\Delta}^{i,N}(s)-X_{s}^{i,N}|^p+\theta^p\Delta^p|b_{\Delta}(Y_{\Delta}^{i,N}(s),Y_{\Delta}^{i,N}(s-\tau),\mu^{Y,N}(s))|^p]{\rm d}s\\
 &+C\Delta^{1-\alpha}\mathbb{E}\int_{0}^t|\Lambda(s)|^{p-2}[(1+|Y_{\Delta}^{i,N}(s)|^l+|Y_{\Delta}^{i,N}(s-\tau)|^l+|X_{s}^{i,N}|^l+|X_{s-\tau}^{i,N}|^l)\\
 &\times(|Y_{\Delta}^{i,N}(s)-X_{s}^{i,N}|+|Y_{\Delta}^{i,N}(s-\tau)-X_{s-\tau}^{i,N}|)+\mathbb{W}_2(\mu^{Y,N}(s),\mu_{s}^{X,N})]{\rm d}s\\
 \leq &C\int_{0}^t\mathbb{E}\left(\sup_{0\leq u\leq s}|Y_{\Delta}^{i,N}(u)-X_{u}^{i,N}|^p\right){\rm d}s+C\Delta^{(1-\alpha)(1+p)}.
 \end{align*}
Further, the Young's inequality, Assumption \ref{sigmamanzudechushitiaojian}, and Lemmas \ref{bdeltayoujie}, \ref{discreteandcontinuoustamedconvergence} lead to
\begin{align*}
&\mathbb{E}\bigg(\sup_{0\leq u\leq t}Z_{4}(u)\bigg)+\mathbb{E}\bigg(\sup_{0\leq u\leq t}Z_{5}(u)\bigg)\\
\leq&C\mathbb{E}\int_{0}^t|\Lambda(s)|^p{\rm d}s+C\mathbb{E}\int_{0}^t(t-s)^{2H-1}\mathbb{W}_2^p(\bar{\mu}^{Y,N}(s),\mu^{Y,N}(s)){\rm d}s\\
&+C\mathbb{E}\int_{0}^t(t-s)^{2H-1}\mathbb{W}_2^p(\mu^{Y,N}(s),\mu_{s}^{X,N}){\rm d}s\\
\leq &C\int_{0}^t\mathbb{E}\left(\sup_{0\leq u\leq s}|Y_{\Delta}^{i,N}(u)-X_{u}^{i,N}|^p\right){\rm d}s+C\int_{0}^t(t-s)^{2H-1}\mathbb{E}\bigg(\sup_{0\leq u\leq s}|Y_{\Delta}^{i,N}(u)-X_{u}^{i,N}|^p\bigg){\rm d}s\\
&+C\Delta^{(1-\alpha)p}+C\Delta^{[(1-\alpha)\wedge H]p}.
\end{align*}
In the similar way as achieving \eqref{J3}, the Young's inequality results in
\begin{align*}
&\mathbb{E}\bigg(\sup_{0\leq u\leq t}|Z_{6}(u)|^p\bigg)\le C\Delta^{(1-\alpha)p}+C\Delta^{[(1-\alpha)\wedge H]p}\\
&+C\epsilon\mathbb{E}\bigg(\sup_{0\leq u\leq t}|\Lambda(u)|^p\bigg)+C\int_{0}^t\mathbb{E}\left(\sup_{0\leq u\leq s}|Y_{\Delta}^{i,N}(u)-X_{u}^{i,N}|^p\right){\rm d}s.
\end{align*}
Collecting the above estimates $Z_{1}-Z_{6}$ together
\begin{align*}
 \mathbb{E}\bigg(\sup_{0\leq u\leq t}|\Lambda(u)|^p\bigg)\leq& C\int_{0}^t\mathbb{E}\bigg(\sup_{0\leq u\leq s}|Y_{\Delta}^{i,N}(u)-X_{u}^{i,N}|^p\bigg){\rm d}s\\
 &+C\int_{0}^t(t-s)^{2H-1}\mathbb{E}\left(\sup_{0\leq u\leq s}|Y_{\Delta}^{i,N}(u)-X_{u}^{i,N}|^p\right){\rm d}s+C\Delta^{\alpha p}.
\end{align*}
Recalling the definition of $\Lambda(t)$ and using Assumption \ref{zhonglixiangdetiaojian}, we may compute
\begin{align*}
  |\Lambda(t)|^p\geq& 2^{1-p}|Y_{\Delta}^{i,N}(t)-X_{t}^{i,N}-D(Y_{\Delta}^{i,N}(t-\tau))+D(X_{t-\tau}^{i,N})|^p\\
  &-|\theta b_{\Delta}(Y_{\Delta}^{i,N}(t),Y_{\Delta}^{i,N}(t-\tau),\mu^{Y,N}(t))\Delta|^p\\
  \geq &2^{1-p}[2^{1-p}|Y_{\Delta}^{i,N}(t)-X_{t}^{i,N}|^p-|D(Y_{\Delta}^{i,N}(t-\tau))-D(X_{t-\tau}^{i,N})|^p]\\
  &-|\theta b_{\Delta}(Y_{\Delta}^{i,N}(t),Y_{\Delta}^{i,N}(t-\tau),\mu^{Y,N}(t))\Delta|^p\\
  \geq& 2^{2-2p}|Y_{\Delta}^{i,N}(t)-X_{t}^{i,N}|^p-2^{1-p}\lambda^p|Y_{\Delta}^{i,N}(t-\tau)-X_{t-\tau}^{i,N}|^p\\
  &-|\theta b_{\Delta}(Y_{\Delta}^{i,N}(t),Y_{\Delta}^{i,N}(t-\tau),\mu^{Y,N}(t))\Delta|^p.
\end{align*}
This gives from Lemmas \ref{bdeltayoujie}, \ref{tamedmomentbound} that
\begin{align*}
&\mathbb{E}\bigg(\sup_{0\leq u\leq t}|X_u^{i,N}-Y_{\Delta}^{i,N}(u)|^p\bigg)\leq  \mathbb{E}\bigg(\sup_{0\leq u\leq t}|\Lambda(u)|^p\bigg)+C\Delta^{(1-\alpha)p}\\
\leq &C\Delta^{\alpha p}+C\int_{0}^t\mathbb{E}\bigg(\sup_{0\leq u\leq s}|Y_{\Delta}^{i,N}(u)-X_{u}^{i,N}|^p\bigg){\rm d}s\\
&+C\int_{0}^t(t-s)^{2H-1}\mathbb{E}\left(\sup_{0\leq u\leq s}|Y_{\Delta}^{i,N}(u)-X_{u}^{i,N}|^p\right){\rm d}s.
\end{align*}
Since the function $E_{2H,1}((\Gamma(2H))^{1/2H}t)$ converges, we can utilize Lemma \ref{Hlemma} to get the assertion.
 \end{proof}

 \begin{thm}
{\rm Let Assumptions \ref{chushizhixidetiaojian}-\ref{sigmamanzudechushitiaojian} hold. Fix any $\bar{p}>4(l+1)$. Then, for $p\in[2,\frac{\bar{p}}{2(l+1)})$, it holds that
\begin{align*}
\mathbb{E}\bigg(\sup_{t\in[0,T]}|X_{t}^i-Y_{\Delta}^{i,N}(t)|^p\bigg)\leq C\times\left\{
\begin{array}{ll}
 N^{-\frac{1}{2}}+\Delta^{\alpha p},\ \ p>\frac{d}{2},\\
     N^{-\frac{1}{2}}\log(1+N)+\Delta^{\alpha p},\ p=\frac{d}{2},\\
     N^{-\frac{p}{d}}+\Delta^{\alpha p},\ p\in[2,\frac{d}{2}),
\end{array}
\right.
\end{align*}
where $C$ is a constant independent of $N$ and $\Delta$.}
\end{thm}
\begin{proof}
  By Lemmas \ref{propagationofchaosholdfinitemoment} and \ref{strongconvergenceofxiN-YiN}, the result is obvious.
\end{proof}

\begin{rem}
{\rm This paper introduces a tamed theta EM scheme for a class of McKean-Vlasov SDEs driven by fractional Brownian motions. In contrast to the existing literature, the proposed framework generalizes several tamed EM schemes studied in works such as \cite{GGHY24} and \cite{HBF23}. The convergence rate is rigorously derived by employing fractional It\^{o} calculus, the fractional BDG inequality, and other pertinent inequalities. This methodology provides a new perspective for handling this category of fractional McKean-Vlasov equations.
}
\end{rem}

\end{document}